\documentclass[reqno]{amsart}

\usepackage[margin=1.9cm]{geometry}

\usepackage{ifthen}       % used for ``draft'' mode
\usepackage{xparse}       % special behaviour for emtpy optional argument

\usepackage{multirow}
\usepackage{colortbl,hhline}

\usepackage{hyperref}

\usepackage{graphicx}                % for graphics
\usepackage{tikz}

\usepackage{amsmath, amssymb}%, theorem}
\usepackage{amsthm}
\usepackage{amscd}

\usepackage{mathtools}

\usepackage{bbm}         % blackboard (see: /usr/share/texmf/mytex/bbm.dvi)
\usepackage{bm}          % used by John for boldface math

\usepackage{mathrsfs}    % use ``real'' caligraphic letters for \mathcal

\usepackage{dsfont}

\usepackage{caption}
\usepackage{subcaption}

\usepackage{float}
\usepackage{tikz-cd}

\usepackage{hyperref}
\usepackage{cleveref}

\newcommand{\mythmstyle}{plain}
\newcommand{\mydefstyle}{definition}

\newcommand{\lapl}[2][{}]{\Delta_{{#2}}^{{#1}}} % symbol for Laplacian:

\usepackage{nicefrac}

\newcommand{\lessWithNumber}[1]{\stackrel{#1}\preccurlyeq}

\NewDocumentCommand{\less}{o}{%
  \IfNoValueTF{#1}
    {\preccurlyeq}
    {\lessWithNumber{#1}}%
}

\ifthenelse{\isundefined \noTikZpictures}
{}
{
  \usepackage{environ}

  \NewEnviron{tikzpicture}{%
    \begin{pgfpicture}
      \pgfpathrectanglecorners{\pgfpointorigin}{\pgfpoint{3cm}{3cm}}%
      \pgfusepath{stroke}
    \end{pgfpicture}%
  }
}

\ifthenelse{\isundefined \draft }
{
  \newcommand{\look}[1]{}%   final version
  \newcommand{\lookO}[1]{}%
  \newcommand{\lookF}[1]{}%
  \newcommand{\lookJ}[1]{}%
}
{
  \usepackage[notref,notcite]{showkeys}
  
  \newcommand{\markerO}{\fbox{\rule{0pt}{0.1ex}\textbf{Olaf}}}
  \newcommand{\markerF}{\fbox{\rule{0pt}{0.1ex}\textbf{Fernando}}}
  \newcommand{\markerJ}{\fbox{\rule{0pt}{0.1ex}\textbf{John}}}
  \newcommand{\look}[1]{\markerO \textbf{*}%\rule{3ex}{2.5mm}
    \footnote{ #1 }}
  \newcommand{\lookO}[1]{\markerO\textbf{*}%\rule{3ex}{2.5mm}
    \footnote{\textbf{Olaf:} #1 }}
  \newcommand{\lookF}[1]{\markerF\textbf{*}%\rule{3ex}{2.5mm}
    \footnote{\textbf{Fernando:} #1 }}
  \newcommand{\lookJ}[1]{\markerJ\textbf{*}%\rule{3ex}{2.5mm}
    \footnote{\textbf{John:} #1 }}

  \usepackage{datetime}
  
  \yyyymmdddate   % use ISO format for date
  \newcommand{\draftInfo}{{\hfill \tiny \ \today, \currenttime, \emph{File:}  \texttt{\jobname.tex \quad}}}

 \let\myTitle\title
 \renewcommand*{\title}[1]{\myTitle{#1 \newline\draftInfo}}
}
\newcommand{\quadtext}[1]{\quad\text{#1}\quad}
\newcommand{\qquadtext}[1]{\qquad\text{#1}\qquad}
\newcommand{\itemref}[1]{\noindent(\ref{#1})}

\numberwithin{equation}{section}

\theoremstyle{\mythmstyle}       % my style (new fonts) -- body italics

\newtheorem{theorem}{Theorem}[section]
\newtheorem*{theorem*}{Theorem}
\newtheorem{proposition}[theorem]{Proposition}
\newtheorem{lemma}[theorem]{Lemma}
\newtheorem{corollary}[theorem]{Corollary}

\theoremstyle{\mydefstyle}        % my style (new fonts) -- body roman
\newtheorem{definition}[theorem]{Definition}

\newtheorem{example}[theorem]{Example}

\newtheorem{remark}[theorem]{Remark}

\newtheorem*{remark*}{Remark}
\newtheorem*{commentolaf*}{Comment Olaf}

\newcommand{\Sec}[1]{Section~\ref{sec:#1}}

\newcommand{\Subsec}[1]{Subsection~\ref{subsec:#1}}

\newcommand{\Fig}[1]{Figure~\ref{fig:#1}}

\newcommand{\Thm}[1]{Theorem~\ref{thm:#1}}

\newcommand{\Lemenum}[2]{Lemma~\ref{lem:#1}~(\ref{#2})}

\newcommand{\Cor}[1]{Corollary~\ref{cor:#1}}

\newcommand{\Prp}[1]{Proposition~\ref{prp:#1}}

\newcommand{\Rem}[1]{Remark~\ref{rem:#1}}

\newcommand{\Remenum}[2]{Remark~\ref{rem:#1}~(\ref{#2})}

\newcommand{\Def}[1]{Definition~\ref{def:#1}}

\newcommand{\abs}[2][{}]{\lvert{#2}\rvert_{{#1}}}    % abs value
\newcommand{\abssqr}[2][{}]{\lvert{#2}\rvert^2_{#1}} % abs squared
\newcommand{\normsymb}{\|}

\newcommand{\normsqr}[2][{}]{\normsymb{#2}\normsymb^2_{#1}} % norm squared

\newcommand{\set}[2]{\{ \, #1 \, | \, #2 \, \} }      % set { #1 | #2 }

\newcommand{\Bigset}[2]{\Bigl\{ \, #1 \, \Bigl|\Bigr. \, #2 \, \Bigr\} }

\makeatletter
\@ifpackageloaded{stix}{%
}{%
  \DeclareFontEncoding{LS2}{}{\noaccents@}
  \DeclareFontSubstitution{LS2}{stix}{m}{n}
  \DeclareSymbolFont{stix@largesymbols}{LS2}{stixex}{m}{n}
  \SetSymbolFont{stix@largesymbols}{bold}{LS2}{stixex}{b}{n}
  \DeclareMathDelimiter{\lMset}{\mathopen} {stix@largesymbols}{"E8}%
                                            {stix@largesymbols}{"0E}
  \DeclareMathDelimiter{\rMset}{\mathclose}{stix@largesymbols}{"E9}%
                                            {stix@largesymbols}{"0F}
}
\makeatother
\renewcommand{\lMset}{\lbrace}
\renewcommand{\rMset}{\rbrace}
\newcommand{\map}[3]{ #1 \colon #2 \longrightarrow #3}    % maps
\newcommand{\card}[1]{\lvert#1\rvert}   % from AMS proceedings file
\DeclareMathOperator{\tr}     {tr}  % trace

\newcommand{\specsymb} {\sigma} % symbol for spectrum

\newcommand{\spec}[2][{}]   {\specsymb_{\mathrm{#1}}(#2)}

\renewcommand{\phi}{\varphi}   % shortcut

\newcommand{\R}{\mathbb{R}} % symbol for real numbers
\newcommand{\N}{\mathbb{N}} % symbol for natural numbers
\newcommand{\wt}{\widetilde}           % shortcut
\newcommand{\wh}{\widehat}          % shortcut
\newcommand{\lsymb}    {\ell}          % symbol for int l-spaces

\newcommand{\lpspace}[1][p]    {\lsymb_{#1}}     % symbol for int L-spaces
\newcommand{\lsqrspace}    {\lpspace[2]}          % symbol for int l-spaces
\newcommand{\lsqr}[2][{}]{\lsqrspace^{#1}({#2})}   % l_2(#1)-spaces

\title{Isospectral graphs via spectral bracketing}%

\author{John Stewart Fabila-Carrasco} %
\address{School of Engineering, Institute for Digital Communications,
  University of Edinburgh, Edinburgh EH9 3FB, U.K. }
\email{John.Fabila@ed.ac.uk}

\author{Fernando Lled\'o} %
\address{Department of Mathematics, University Carlos III de Madrid,
  Avda. de la Universidad 30, 28911 Legan\'es (Madrid), Spain and
  Instituto de Ciencias Matem\'aticas (CSIC-UAM-UC3M-UCM), Madrid}
\email{flledo@math.uc3m.es}

\author{Olaf Post} %
\address{Fachbereich 4 -- Mathematik, Universit\"at Trier, 54286
  Trier, Germany} \email{olaf.post@uni-trier.de}

\date{\today}

\thanks{JSFC was supported by the Leverhulme Trust via a Research
  Project Grant (RPG-2020-158). FLl was supported by the Severo
  Ochoa Programme for Centres of Excellence in R\&D (SEV-2015-0554)
  and from the Spanish National Research Council, through the
  \textit{Ayuda extraordinaria a Centros de Excelencia Severo Ochoa}
  (20205CEX001) and by the Madrid Government under the Agreement with UC3M in the line of \emph{Research Funds for Beatriz Galindo Fellowships} (C\&QIG-BG-CM-UC3M), and in the context of the V PRICIT}


\begin{document}
%-------------------------------------------------------------
% Abstract.
%-------------------------------------------------------------
\begin{abstract}
  In this article, we develop a perturbative technique to construct
  families of non-isomorphic discrete graphs which are isospectral for
  the standard (also called normalised) Laplacian and its signless
  version. We use vertex contractions as a graph perturbation and
  spectral bracketing with auxiliary graphs which have certain
  eigenvalues with high multiplicity. There is no need to know
  explicitly the eigenvalues or eigenfunctions of the corresponding
  graphs. We illustrate the method by presenting several families of
  examples of isospectral graphs including fuzzy complete bipartite
  graphs and subdivision graphs obtained from the previous examples.
  All the examples constructed turn out to be also isospectral for the
  standard (Kirchhoff) Laplacian on the associated equilateral metric
  graph.
\end{abstract}

%-----------------------------------------------------------------------------

\maketitle

%-----------------------------------------------------------------------
%
%
\section{Introduction}
\label{sec:intro}
%
%-----------------------------------------------------------------------

Perturbative techniques in the study of linear Hilbert space operators
(or matrices) have a long and fertile history (see,
e.g.~\cite{kato:95,baumg:85}). From a general point of view, they
provide a method of investigation based
on the knowledge of the spectrum of a free (non-perturbed) operator
(typically a Laplacian on some domain of $\R^n$ or on a graph) and the
analysis of its stability or variation under a controlled perturbation
expressed in terms of a potential or a geometrical manipulation of the
underlying spaces (see, e.g.~\cite{so:99,lledo-post:07}).
In~\cite{fclp:22a} we presented a systematic study of the spectral
effects of graph perturbations (like edge deletion, vertex
contraction, vertex virtualisation etc.) for arbitrary weights and
including arbitrary magnetic potentials. These effects can be
quantified in terms of a natural spectral preorder. In this article,
we present a perturbative method leading to the construction of
isospectral graphs which are interpreted as perturbations (via vertex
merging) of auxiliary graphs which have certain eigenvalues with high
enough multiplicity.

Spectral graph theory studies the relationship between the structure
of a graph and the spectrum of operators naturally associated with
the graph which are usually represented by matrices with respect to a
suitable orthonormal basis.  In this article, we consider only simple
finite graphs, i.e., there are no multiple edges nor loops.
For a graph $G=(V,E)$, one can define different natural operators
associated with it.  Typically, one considers the \emph{adjacency
  matrix} $A_G$ with respect to a numbering of the vertices
$V=\{v_1,\dots,v_n\}$ having entries $1$ and $0$ depending on whether
there is an edge between two vertices or not.  The \emph{combinatorial
  Laplacian} is given by $D_G-A_G$, where $D_G$ is the diagonal matrix
with the degree of each vertex on the diagonal.  The \emph{signless
  combinatorial Laplacian} is defined by $D_G+A_G$.  Another well
known choice is the \emph{standard Laplacian} (also called
\emph{geometric} or \emph{normalised} in the literature,
cf. \Remenum{DML}{std.vs.normalised}), represented by
\begin{equation*}
%  \label{eq:std.lapl}
  \lapl G
  = D_G^{-1/2}(D_G - A_G)D_G^{-1/2}
  = I_n - D_G^{-1/2}A_G D_G^{-1/2},
\end{equation*}
where $I_n$ is the identity $(n\times n)$-matrix (see
also~\cite{chung:97} for additional references and motivation).  The
\emph{signless standard Laplacian} is defined as
$\lapl {G^+} = I_n + D_G^{-1/2}A_G D_G^{-1/2}$ (see also
\Subsec{mw.lapl}).  Both standard Laplacians have their spectrum
contained in the interval $[0,2]$.

An important question with many applications is whether the spectrum
of a certain operator on a graph determines the graph up to
isomorphism or not. Two non-isomorphic graphs having the same spectrum
for a certain operator are called \emph{isospectral graphs} with
respect to that operator.  Two non-isomorphic graphs such that their
adjacency matrices have the same spectrum are often referred to as
\emph{cospectral} cf.~e.g.~\cite{cddgt:88}, i.e., isospectral with
respect to the adjacency matrix in our notation. Often, classical
problems related to isospectrality are considered for the
\emph{combinatorial} Laplacian, see
e.g.~\cite{cds:95,brouwer-haemers:12,hae:03}. Much less is known for
the standard Laplacian, as most previous works concentrate on the
adjacency matrix or the combinatorial Laplacian.  We refer in this
article to \emph{isospectral graphs} as \emph{isospectral} with
respect to the \emph{standard Laplacian} (or its signless version).
In recent time, though, results and constructions of isospectral
graphs for the standard Laplacian have been considered by several
authors. In~\cite{butler-grout:11}, the authors construct from given
isospectral bipartite graphs (e.g.~unions of complete bipartite
graphs) new isospectral graphs.  The proof is based on a concrete
construction of eigenfunctions of the new graphs from the old ones and
checking that all eigenvalues are exhausted.  Examples where
isospectral graphs with \emph{different} number of edges are present
are constructed in~\cite{butler:15}, based on so-called \emph{twin
  vertices} and simultaneously scaling the edge and vertex
weights. Other recent references describing isospectral graphs for the
standard Laplacian are given in \cite{das:16,lorenzen:22,butler:10}.

In this article we present a new perturbative technique to construct
isospectral graphs of the standard Laplacian (including its signless
version).  This new philosophy is based on the elementary perturbation
of a graph and the control of the effect of this perturbation on some
of its eigenvalues. From known multiplicity of certain
eigenvalues of some auxiliary graphs (and \emph{no} knowledge of the
corresponding eigenfunctions), we construct isospectral graphs using
spectral bracketing, i.e., enclosing unknown eigenvalues between the
corresponding eigenvalues of the auxiliary graphs. Concretely, a
spectral bracketing of a graph $G$ with another graph $G'$ means
that the eigenvalues interlace, i.e., that for some $t\in\N$ the
relations
\begin{equation}
  \label{eq:ev.interlacing}
  \lambda_1(G') \le \lambda_1(G) \le \lambda_{1+t}(G'), \quad
  \lambda_2(G') \le \lambda_2(G) \le \lambda_{2+t}(G'), \quad
  \dots, \quad
  \lambda_n(G')\le \lambda_n(G) \le \lambda_{n+t}(G')
\end{equation}
hold, where $n=\card G$ and $n+t=\card{G'}$ are the orders of $G$
and $G'$, respectively. Symbolically, we denote this as
$G' \less G \less[t] G'$ for this eigenvalue interlacing, where
$t\in\N_0$ denotes the \emph{spectral shift} in the discrete label
counting the eigenvalues that fixes the bracketing, i.e.,
$\lambda_k(G)\in [\lambda_k(G'),\lambda_{k+t}(G')]$, $k=1,\dots,n$. If
$G$ now has an eigenvalue of multiplicity $\mu > t$, then $G'$ also
has this eigenvalue with multiplicity at least $\mu-t>0$. In
\Rem{spec.brak.visualised} we also give a diagrammatic visualisation
of the spectral bracketing method. See also \cite{chen:04} for an
interlacing result under edge deletion.

This technique is more efficient the smaller the spectral shift $t$
is, since then, the bracketing becomes tighter. Ideally one looks for
elementary perturbations of the graph that have $t=1$ in only one of
the spectral relations leading to the classical interlacing of
eigenvalues.  In this sense, the spectral shift quantifies the
\emph{spectral cost} of the perturbation. This cost depends much on
which weights are used on the graph. For example, an edge deletion has
a minimal spectral cost for combinatorial weights but not for standard
weights while a single vertex contraction is a convenient perturbation
of the graph with standard weights but not for combinatorial weights
(see \cite[Corollaries~4.2 and 4.7]{fclp:22a}).  For this reason we
will focus in this article on perturbations where the graph $\wt G$ is
obtained from $G$ by contracting vertices of $G$ such that the
difference of the number of vertices is given by
$t=\card{G}-\card{\wt G}$ and the spectral bracketing of this
operation is given by
\begin{equation*}
  G \less \wt G \less[t] G.
\end{equation*}
The task is now to find suitable auxiliary graphs $G$ having
eigenvalues with high multiplicity such that one can exhaustively
determine eigenvalues $\wt G$ by the respective bracketings. This can
be done, for example, if the different eigenvalue sets of the
auxiliary graphs are disjoint (and hence the corresponding eigenspaces
must be orthogonal). If two non-isomorphic graphs $\wt G_1$ or
$\wt G_2$ can be sandwiched by the same auxiliary graphs as above and
if all eigenvalues are exhausted, then they are isospectral by
\Thm{main.idea}.  In the final section we apply this method to
different classes of examples where one knows the spectrum of the
corresponding auxiliary graphs and, hence, one can also determine the
spectrum of the corresponding isospectral graphs.  But the
construction technique of isospectral graphs proposed in this article
\emph{only requires to know the multiplicity of certain eigenvalues}
of the auxiliary graphs.

This article is structured as follows: \Sec{graphs} contains a brief
introduction of the notation and results on graphs with standard
weights and their Laplacians.  In the following section the main tools
for the analysis are presented, the spectral bracketing of graphs and
vertex contraction as the main geometrical perturbation of the graph.
It is proved that vertex contraction has small spectral cost for the
standard weights.  \Thm{main.idea} is the main result in this section
and shows how the multiplicity of certain eigenvalues of the auxiliary
graphs may lead to the determination of the spectrum of the sandwiched
graph. This idea is exploited in \Sec{examples} to present different
families of isospectral and non-isomorphic graphs. These examples
include fuzzy ball and fuzzy complete graph families. Finally, new
examples obtained as subdivision graphs are shown to be isospectral
using directly the diagrammatic picture associated to different
spectral bracketings. All the examples of isospectral discrete graphs
can be naturally turned into equilateral metric graphs which, again,
are isospectral for the corresponding (unbounded) self-adjoint
standard Laplacian (cf.~\cite{berko-kuch:13}).

%-----------------------------------------------------------------------
%
%
\section{Discrete graphs and their standard Laplacians}
\label{sec:graphs}
%
%-----------------------------------------------------------------------

In this section, we introduce briefly the discrete structures and
operations needed for the construction for families of isospectral
graphs. We will only consider here discrete, simple and unoriented
finite graphs together with its Laplacian and signless Laplacian with
standard weights.  We refer to \cite{fclp:22a} for a general analysis
in the context of directed multigraphs with an additional magnetic
potential described as a $S^1$-valued function on the edges.

%-----------------------------------------------------------------------
\subsection{Discrete graphs}
\label{subsec:disc.graphs}
%-----------------------------------------------------------------------

A \emph{(simple) graph} $G$ consists of a finite \emph{vertex} set
$V$ and an adjacency relation on $V$ (i.e., a symmetric and
non-reflexive relation on $V$); two vertices being in this relation
are called \emph{adjacent} and $\{u,v\}$ is called an \emph{edge}
joining $u$ and $v$. The \emph{order} of a finite graph $G$ is the
cardinality of its vertex set which we denote by
$\card G:=\card{V}$. We define the neighbourhood of a vertex $v$ as
\begin{equation*}
  N_v := \set{u \in V}{\text{$u$ and $v$ are adjacent}}
\end{equation*}
and the \emph{degree} of $v$ is the number of vertices adjacent with
$v$ which we denote as
\begin{equation*}
  \deg v := \card {N_v}.
\end{equation*}
Since we focus in this article on Laplacians with standard weights we
exclude \emph{isolated} vertices (i.e., vertices of degree $0$ or,
alternatively, we assume $\deg_G v>0$ for all $v \in V$).  A vertex of
degree $1$ is called \emph{pendant}.  Let $G$ be a graph with $n$
vertices labelled as $v_1, v_2, \dots, v_n$ and such that
$\deg v_1 \le \deg v_2 \le \dots \le \deg v_n$.  The list
\begin{equation*}
%  \label{eq:deg.set}
  \deg G := \bigl( \deg v_1, \deg v_2,\dots, \deg v_n \bigr)
\end{equation*}
is called the \emph{degree list} of $G$ and provides a graph
invariant, i.e., if the degree sequence or the degree list of two
graphs differ, then the graphs cannot be isomorphic.  A graph $G$ is
\emph{connected} if for any two vertices $u,v$ there exists a path
from $u$ to $v$, i.e., there is a sequence of distinct vertices
$v_0,v_1,\dots, v_n$ such that $v_{i-1}$ and $v_i$ are adjacent for
$1\leq i\leq n$ and such that $u=v_0$ and $v=v_n$. If the explicit
dependence of the graph $G$ is needed, we write $V=V(G)$ or $\deg^G v$
etc.

A central operation needed in this article is the
\emph{contraction of vertices} (also the name \emph{gluing}
or \emph{merging} is used in the literature). Recall that
the choice of a subset of vertices $V_0\subset V$ naturally
specifies an equivalence relation of vertices. In this article
the result of contracting vertices must give, again, a simple graph.
We refer to Definition~2.1 and Remark~2.2 \cite{fclp:22a} for a general
description of contractions in the context of multiple directed graphs.

\newcommand{\simContr}{\sim}
\newcommand{\quotient}[2][]{#2/{\sim_{#1}}}
%-----------------------------------------------------------------------
\begin{definition}[Contracting vertices]
  \label{def:contraction}
  Let $G$ be a graph with vertex set $V$, consider an equivalence
  relation $\simContr$ on the vertex set $V$ and denote by $[v]$ the
  class of vertices related with $v$.  The quotient graph
  $\quotient G$ is the graph with vertex set
  $\quotient V=\set{[v]}{v \in V}$ and keeping all edges.

  We say that $\quotient G$ is obtained from $G$ by \emph{contracting
    the vertices} (with respect to $\simContr$). Two vertices $[u]$
  and $[v]$ in $\quotient G$ are adjacent if $[u]\not= [v]$ and if
  some vertex in the class of $[v]$ is adjacent in $G$ to some vertex
  in the class of $[u]$.

  The \emph{shrinking number} $r$ of the equivalence relation
  $\simContr$ is defined by
  $r:=\card{G}-\card{\quotient G}=\card V-\card{\quotient V}$ and
  quantifies the reduction of vertices in the quotient graph.
\end{definition}
%-----------------------------------------------------------------------

In this article we will only allow contractions that respect the
category of simple graphs.  In particular, the quotient graph must be
simple as well.  For example, adjacent vertices in $G$ can not be
contracted since the quotient graph is not allowed to have loops in
this article.

%-----------------------------------------------------------------------
\subsection{Standard Laplacians}
\label{subsec:mw.lapl}
%-----------------------------------------------------------------------

For a finite graph $G$, we associate the following
natural weighted Hilbert space
\begin{align*}
  \lsqr{V,\deg}
  :=\Bigset{\map f V \R}
    {\normsqr[\lsqr{V,\deg}] f
    = \sum_{v \in V} \abssqr{f(v)} \deg v}.
\end{align*}
The standard Laplacian is defined as follows (for an approach
using discrete exterior derivatives, more general weights and magnetic
potentials we refer to~\cite[Sec.~3]{fclp:22a} and references therein):

%-----------------------------------------------------------------------
\begin{definition}[standard Laplacian and standard signless Laplacian]
  \label{def:DML}
  \begin{subequations}
    The \emph{standard (discrete) Laplacian} of a graph $G$ is the
    self-adjoint operator
    $\map{\lapl G}{\lsqr{V,\deg}}{\lsqr{V,\deg}}$ defined for
    $v \in V$ and $f \in \lsqr{V,\deg}$ by
    \begin{equation}
      \label{eq:DML}
      \bigl(\lapl G f \bigr)(v)
      = \frac1{\deg v} \sum_{u \in N_v}
      \bigl(f(v)-f(u)\bigr)
      = f(v)-\frac1{\deg v } \sum_{u \in N_v} f(u).
    \end{equation}
    If $G$ has order $n$, then we write the spectrum of the
    corresponding Laplacian $\lapl G$ as the list
    \begin{equation*}
      \spec G
      = \bigl(\lambda_1(G), \lambda_2(G), \dots ,\lambda_n(G)\bigr),
    \end{equation*}
    where the eigenvalues $\lambda_1(G), \dots, \lambda_n(G)$ are
    written in ascending order and repeated according to their
    multiplicities.

    Similarly, we define the \emph{signless standard Laplacian}
    $\lapl{G^+}$ for $v \in V$ and $f \in \lsqr{V,\deg}$
    by\footnote{Note that in~\cite{fclp:18}
      and~\cite[Sec.~5.3]{fclp:22a} we reserved the superscript $^+$
      to denote Dirichlet Laplacians. Since in this article we do not
      need the operation of vertex virtualisation we use the
      superscript $^+$ in a different way: the results in the last
      section can be written in a much more convenient way using $^-$
      for the Laplacian and $^-$ for the signless Laplacian.}
    \begin{equation}
      \label{eq:DML+}
      \bigl(\lapl{G^+} f \bigr)(v)
      = \frac1{\deg v} \sum_{u \in N_v}
      \bigl(f(v)+f(u)\bigr)
      = f(v)+\frac1{\deg v } \sum_{u \in N_v} f(u)
    \end{equation}
    Its eigenvalues are denoted by
    $\spec {G^+} = \bigl(\lambda_1(G^+), \lambda_2(G^+), \dots
    ,\lambda_n(G^+)\bigr)$.
  \end{subequations}
\end{definition}
%-----------------------------------------------------------------------

We refer to \cite{lledo-post:08b} for a description of the unoriented
homology related to the signless Laplacian and to \cite[Sec.~2.4 and
Sec.~3.1]{fclp:22a} for the use of the magnetic potential as an
interpolation parameter connecting both Laplacians. More applications
of the magnetic potential to analyse combinatorial properties of the
graph are given in~\cite{fclp:22b}.

%-----------------------------------------------------------------------
\begin{remark}
  \label{rem:DML}
  \indent
  \begin{enumerate}
  \item \label{std.vs.normalised} Some authors call the Laplacian with
    degree weight \emph{geometric}, as some results referring to this
    Laplacian are closer to the Laplacian on a manifold or a metric
    graph.  Sometimes, the standard Laplacian is also called
    \emph{normalised}, but we find that this name should better be
    reserved for edge weights $w_e$ with edge $e=\{u,v\}$ and
    associated vertex weights
    $\deg^w v := \sum_{u \in N_v} w_{\{u,v\}}$.
  \item
    \label{spec.0.2}
    Recall that the spectra of the previous Laplacians is contained
    in the interval $[0,2]$, i.e., $\spec
    {G^\pm}\subset[0,2]$. Moreover, we always have $0 \in \spec G$,
    and $0$ is simple if and only if $G$ is connected.  Similarly, we
    have $2 \in \spec {G^+}$.  The eigenfunction of both eigenvalues
    is a constant function on $V$ as one easily sees
    from~\eqref{eq:DML} and \eqref{eq:DML+}.

  \item
    \label{matrix.rep}
    With respect to the orthonormal basis $\set{\delta_v}{v \in V}$
    with $\delta_v(u)=(\deg v)^{-1/2}$ if $u=v$ and $\delta_v=0$ if
    $u \ne v$, the standard Laplacian denoted with $(-)$ as $\lapl G=\lapl {G^-}$ and the standard signless
    Laplacian denoted with $(+)$ as $\lapl{G^+}$ are represented by the matrices with entries
    \begin{equation}
      \label{eq:DML'}
      (\lapl {G^\mp})_{u,v}
      =
      \begin{cases}
        1, & u=v,\\
        \mp \bigl((\deg u)(\deg v)\bigr)^{-1/2},&
        \text{$u$ and $v$ are adjacent,}\\
        0, & \text{otherwise\;.}
      \end{cases}
    \end{equation}
  \end{enumerate}
\end{remark}
%-----------------------------------------------------------------------

%-----------------------------------------------------------------------
%
% cccc
\section{Spectral bracketing and graph perturbations}
\label{sec:spectral_preorder}
%
%-----------------------------------------------------------------------

We present in this section our original idea of how to construct
isospectral graphs based on certain graph perturbations whose spectral
effect we can control.  We will sandwich the original graphs between
two auxiliary graphs with high symmetry which will specify the
spectral bracketing. Our method here is based on the concept of ``spectral
preorder'' in relation to vertex contraction (see also \cite[Section~4]{fclp:22a}
for a general spectral analysis of elementary perturbations of multidigraphs).

%-----------------------------------------------------------------------
\subsection{Spectral preorder}
\label{subsec:spec-ord}
%-----------------------------------------------------------------------
In this section we introduce a spectral preorder in the class of simple
graphs with standard weights based on the order relation of
consecutive lists of eigenvalues of the corresponding Laplacians
written in increasing order and repeated according to their
multiplicities. From an algebraic viewpoint the spectral preorder is a
quite flexible generalisation of the eigenvalue interlacing known for
matrices (see e.g.~\cite[Theorem~4.3.28]{horn-johnson:13}
and~\cite[Section~2.5 and 3.2]{brouwer-haemers:12}). This relation can
be generalised to include a shift $t\in\N_0$ in the list of
eigenvalues.  We refer to \cite{fclp:18,fclp:22a} for proofs and additional motivation.
We will use here this preorder to control the spectral effect of
vertex contraction.
%-----------------------------------------------------------------------
\begin{definition}[Spectral preorder $\less$]
  \label{def:spectral-ordering}
  Let $G$ and $G'$ be two graphs and denote the
  corresponding increasing lists of the eigenvalues of the respective
  standard Laplacians, repeated according to their multiplicity, by
  \begin{equation*}
    \bigl(\lambda_1(G), \lambda_2(G), \dots,
    \lambda_{\card{G}}(G)\bigr)
    \qquadtext{and}
    \bigl(\lambda_1(G'), \lambda_2(G'), \dots,
    \lambda_{\card{G'}}(G')\bigr).
  \end{equation*}
  We say that \emph{$G$ is spectrally smaller than $G'$ with shift
    $t \in \N_0$}, and we denote this by
  \begin{equation*}
    G\less[t]G',
  \end{equation*}
  if the following two conditions hold:
  \begin{equation*}
    \card G\ge \card {G'}-t\quadtext{and}
    \lambda_k(G) \le \lambda_{k+t}(G')
    \qquad \text{for all $1 \le k \le \card{G'}-t$.}
  \end{equation*}
  If $t=0$ we write simply $G \less G'$.

  The same definition applies to the signless Laplacian, i.e., we
  write $G^+ \less[t] (G')^+$ if the eigenvalues of the signless
  Laplacians fulfil $\lambda_k(G^+) \le \lambda_{k+t}((G')^+)$ for
  all $1 \le k \le \card{G'}-t$.
\end{definition}
%-----------------------------------------------------------------------

%-----------------------------------------------------------------------
\begin{remark}[spectral preorder visualised]
  \label{rem:spec.ord.visualised}
  We present here the following useful diagrammatic representation of
  the relation $G\less[t]G'$.
  \vspace{2mm}
  \begin{center}
    \begin{tabular}{lp{6ex}p{6ex}p{6ex}|%
      p{6ex}|p{6ex}|p{6ex}|%
      p{6ex}|p{6ex}}%
      \hhline{~~~~-----}
      $G$ & & &
      & $\lambda_1$ & $\lambda_2$ & \dots & $\lambda_{n-1}$
      & \multicolumn{1}{p{6ex}|}{$\lambda_n$}
      \\\hhline{~--------}
      $G'$ &  \multicolumn{1}{|p{6ex}}{$\lambda_1'$}
              &  \multicolumn{1}{|p{6ex}|}{\dots}
                &  \multicolumn{1}{|p{6ex}|}{$\lambda_t'$}
      & $\lambda_{1+t}'$ & $\lambda_{2+t}'$ & \dots
                                          & $\lambda_{n-1+t}'$
      & %\multicolumn{1}{|p{6ex}|}{?}
      \\\hhline{~-------~}
    \end{tabular}
  \end{center}
  \vspace{2mm}
  We list the eigenvalues $\lambda_k=\lambda_k(G)$, $n=\card G$
  and $\lambda_k'=\lambda_k(G')$ of $G$ and $G'$ in the first
  and second row, respectively.  The condition
  $\card G \ge \card{G'} - t$ just means that the last box of the
  second row can not exceed the last box of the first row.  Moreover,
  eigenvalues in the same column increase when going down the column
  and one can easily visualise the spectral relation just looking at
  the successive columns.
\end{remark}
%-----------------------------------------------------------------------

Let $G$, $G'$ and $G''$ be finite graphs. We describe next some
useful conventions and easy consequences of the preceding definition.
We write $G \less[t'] G'$ and $G' \less[t''] G''$ simply by
$G\less[t'] G' \less[t''] G''$, and call such estimates
\emph{spectral bracketing} or \emph{eigenvalue interlacing}.
%-----------------------------------------------------------------------
\begin{remark}[spectral bracketing visualised]
   \label{rem:spec.brak.visualised}
  Later we will often consider the graph sandwiching $G' \less G \less[t]G'$
  which implies that
  $\card G+t= \card{G'} \ge \card{G}$.  If $\card{G}=n$, then
  $G' \less G \less[t] G'$ is equivalent with the
  \emph{interlacing} of the spectra corresponding to the Laplacians of
  $G$ and $G'$ similarly as
  in~\cite[Section~2.5]{brouwer-haemers:12},
  i.e.,~\eqref{eq:ev.interlacing} holds.  Using the diagrammatic
  representation of \Rem{spec.ord.visualised}, we write
  \vspace{2mm}
  \begin{center}
    \begin{tabular}{lp{6ex}p{3ex}p{6ex}|p{6ex}%
      |p{6ex}|p{6ex}%
      |p{3ex}%
      |p{6ex}%
      |p{6ex}|  p{6ex} p{3ex}p{6ex}}
      \hhline{~~~~---------}
     % 1         2 3
      $G'$& & & &
     %4     5     6     7       8       9     10   11    12
      $\lambda'_1$ & $\lambda'_2$ & $\lambda'_3$ & \dots &
      $\lambda'_{n-1}$ & $\lambda'_{n}$ &
      % 13
      \multicolumn{1}{p{6ex}|}{$\lambda'_{n+1}$} &
      % 14
      \multicolumn{1}{p{3ex}|}{\dots} &
      \multicolumn{1}{p{6ex}|}{$\lambda'_{n+t}$}
      \\\hhline{~~~~---------}
      $G$ &&&  & $\lambda_1$ & $\lambda_2$ & $\lambda_3$ & \dots
       & $\lambda_{n-1}$ & $\lambda_{n}$ & &
       \\\hhline{~---------}
       $G'$ &
       \multicolumn{1}{|p{6ex}}{$\lambda_1$} &
       \multicolumn{1}{|p{3ex}|}{\dots}
               & $\lambda'_t$ & $\lambda'_{1+t}$ & $\lambda'_{2+t}$
                                  & $\lambda'_{3+t}$ & %
       \dots &
        $\lambda'_{n-1+t}$ & $\lambda'_{n+t}$
       \\\hhline{~---------}
    \end{tabular}
    \vspace{2mm}
  \end{center}
  The condition $\card{G}=\card {G'}-t$ means here that the second
  and third columns are aligned on the right hand side.  In
  particular, if $t=1$ it becomes the usual eigenvalue interlacing
  (justifying also the name)
  \begin{equation*}
    \lambda_1(G') \le \lambda_1(G) \le \lambda_2(G') \le \lambda_2(G)
    \le \dots \le \lambda_n(G') \le \lambda_n(G) \le \lambda_{n+1}(G').
  \end{equation*}
\end{remark}

The following result is a direct consequence of the definition of
spectral preorder and quantifies the stability of eigenvalues with
high multiplicity under vertex contraction perturbation.
%---------------------------------------------------------------------
\begin{proposition}
  \label{prp:mult.spec}
  Consider two graphs $G,G'$ satisfying the spectral relation
  $G' \less[t'] G \less[t''] G'$.  If $\lambda$ is an eigenvalue of
  the Laplacian on $G'$ with multiplicity $\mu>t'+t''$, then
  $\lambda$ is an eigenvalue of the Laplacian on $G$ with
  multiplicity at least $\mu-(t'+t'')$.
\end{proposition}
%---------------------------------------------------------------------

%-----------------------------------------------------------------------
\subsection{Spectral bracketing and vertex contraction}
\label{subsec:vx.contr.sp}
%-----------------------------------------------------------------------

The following result will be central to the class of examples considered in the
next section. An important observation for our method is that
multiple eigenvalues remain eigenvalues with lower multiplicity in the
quotient graph provided the shrinking number of the quotient is small
enough (cf.~\Def{contraction}). We present next vertex contraction
as a perturbation with low spectral cost for standard weights; see
Theorem~3.14 in~\cite{fclp:22a} for a proof.

%-----------------------------------------------------------------------
\begin{proposition}[contracting vertices and spectral bracketing]
  \label{prp:homomorphism}
  Let $G$ be a finite discrete graph.  If $\wt G=\quotient G$ is a
  graph obtained from $G$ by contracting vertices according to an
  equivalence relation $\simContr$ on $V(G)$, then we have the
  following spectral relation between the corresponding standard
  Laplacians
  \begin{equation*}
    G \less \wt G \less[t] G,
  \end{equation*}
  where $t=\card{ G}-\card{\wt G}\ge 0$ is the shrinking number of
  $\simContr$ (cf.~\Def{contraction}). The same result is true for
  the signless Laplacians, i.e., we have
  $G^+ \less \wt G^+ \less[t] G^+$.
\end{proposition}
%-----------------------------------------------------------------------

Note that \Prp{mult.spec} can be immediately applied to the graph
perturbation introduced before. In fact, the total spectral shift $t$
is the shrinking number of the relation $\sim$
(cf. \Prp{homomorphism}).

%-----------------------------------------------------------------------
\subsection{Spectral bracketing and determination of the spectrum}
\label{subsec:vx.contr.sp-b}
%-----------------------------------------------------------------------

We tacitly understand subsets $\Lambda$ of a spectrum $\spec G$
as \emph{multisets} without formally introducing them.  A simple
workaround is to think of $\Lambda$ as an ordered list
$\Lambda=(\lambda_1,\dots,\lambda_n)$ with
$\lambda_1\le \lambda_2 \le \dots\le \lambda_n$, and multiple members
are repeated according to their multiplicity.  The notation
$\lambda \in \Lambda$ then means that there is $j \in \{1,\dots, n\}$
such that $\lambda=\lambda_j$.  Similarly, we understand
$\Lambda ' \subset \Lambda$, i.e., each $\lambda \in \Lambda'$
appearing $\mu$ times appears at least $\mu$ times also in $\Lambda$.
%-----------------------------------------------------------------------
\begin{definition}[multiplicity and subsets]
  We say that a subset of non-negative numbers $\Lambda$ has
  \emph{multiplicity $\mu$ in $\spec G$} if each element
  $\lambda \in \Lambda$ having multiplicity $s$ in $\Lambda$ has at
  least multiplicity $s \mu$ in $\spec G$.  We denote this relation
  by $\Lambda^{(\mu)}\subset \spec G$.
\end{definition}
%-----------------------------------------------------------------------

Next, we present the key result on which we base our construction of
isospectral graphs in \Sec{examples}. Namely, we specify conditions
that guarantee that the spectrum of a graph $\wt G$ is obtained by
spectrally bracketing the original graph with the auxiliary graphs
$G_j$ having eigenvalues of high multiplicity.

%-----------------------------------------------------------------------
\begin{theorem}
  \label{thm:main.idea}
  Let $J$ be a finite index set and consider for each $j\in J$ the
  finite set of non-negative numbers $\Lambda_j\subset [0,2]$. Assume
  that the graphs $G$ and $G_j$ ($j\in J$) with standard Laplacians
  satisfy the following conditions:
  \begin{enumerate}
  \item
    \label{main.idea.a}
    $G_j \less[t'_j] G \less[t''_j] G_j$, $j\in J$ and denote the
    total spectral shift by $t_j:=t_j'+t_j''$;
  \item
    \label{main.idea.b}
    $\Lambda_j$ has multiplicity $\mu_j$ in $\spec {G_j}$, $j\in J$,
    and $\mu_j> t_j=t_j'+t_j''$;
  \item
    \label{main.idea.c}
    the sets ${\Lambda_j}$, ${j \in J}$ are pairwise disjoint;
  \item the order of $G$ is
    \label{main.idea.d}
    \begin{equation}
      \label{eq:total.number}
      \card{G}
      =\sum_{j \in J}(\mu_j-t_j)\card{\Lambda_j}.
    \end{equation}
  \end{enumerate}
  Then the spectrum of $G$ is determined by the subsets $\Lambda_j$,
  namely we have
  $\spec{G}=\biguplus_{j \in J} \Lambda_j^{(\mu_j-t_j)}$, where
  $\biguplus$ means the disjoint union of multisets.
\end{theorem}
%-----------------------------------------------------------------------
\begin{proof}
  From \Prp{mult.spec} and conditions
  \itemref{main.idea.a}--\itemref{main.idea.b} above we conclude that
  each $\Lambda_j$ has multiplicity at least $(\mu_j-t_j)$ in
  $\spec G$, where the total shift is $t_j:=t_j'+t_j''$. We denote
  this relation as $\Lambda_j^{(\mu_j-t_j)} \subset \spec{G}$ (as
  multiset). Moreover, since $\Lambda_j \cap \Lambda_{j'}=\emptyset$
  for all $j,j' \in J$ with $j \ne j'$ condition \itemref{main.idea.c}
  implies that an eigenvalue of $G$ is contained in only \emph{one}
  of the sets $\Lambda_j$. Finally, condition~\itemref{main.idea.d}
  guarantees that all eigenvalues of $G$ are necessarily in one of
  the sets $\Lambda_j$ and, therefore, the spectrum is exhausted.
\end{proof}
%-----------------------------------------------------------------------

%-----------------------------------------------------------------------
\begin{remark}
  \label{rem:main.idea}
  \indent
  \begin{enumerate}
  \item
    \label{why.disjoint}
    The disjointness condition~\itemref{main.idea.c} of
    \Thm{main.idea} is necessary to guarantee that one
    obtains all \emph{different} eigenvalues, i.e., that eigenvalues
    from different sets $\Lambda_j$ and $\Lambda_{j'}$ for
    $j,j' \in J$ with $j \ne j'$ correspond to \emph{different}
    eigenfunctions.
  \item
    \label{one.ev.missing}
    Sometimes the condition~\itemref{main.idea.d} can be relaxed. For
    example, recall that $0$ (resp. $2$) is always an eigenvalue of
    the Laplacian (resp. signless Laplacian) with constant
    eigenfunction; also, some eigenvalues may possibly be known
    because the graph is bipartite etc.
    \begin{itemize}
    \item It is enough that~\eqref{eq:total.number} adds up to $n-1$,
      where we put $n:=\card {G}$. In fact, the missing eigenvalue
      can be recovered from the remaining ones by the following trace
      relation for the the standard or signless Laplacians
      \begin{equation*}
%        \label{eq:ev.trace}
        \sum_{k=1}^n \lambda_k(G)
        = \tr \lapl {G} = n\;.
      \end{equation*}
      The latter equality follows from the fact that the
      matrix representation of $\lapl{G^\mp}$ has only $1$ on its
      diagonal, cf.~\eqref{eq:DML'}.

    \item Any known eigenvalues of $G$ can be put in a (multi-)set
      $\Lambda_{j_0}$ for some index $j_0 \in J$.  Choosing
      $G_{j_0}=G$, $t_{j_0}'=t_{j_0}''=0$ and $\mu_{j_0}=1$,
      \Thm{main.idea} can formally applied and only
      $\card{\wt G}-\card {\Lambda_{j_0}}$ eigenvalues need to be
      recovered.  We apply this idea sometimes to the eigenvalue $0$
      or to $\{0,2\}$ if the graph is bipartite.  Recall finally, that
      if $G$ is bipartite, one can also use the fact that
      $\lambda \in \spec {G}$ if and only if $2-\lambda \in \spec {G}$
      (cf.~\cite[Lemma~1.8]{chung:97}
      or~\cite[Proposition~2.3]{lledo-post:08b}).
    \end{itemize}
  \end{enumerate}
\end{remark}

%-----------------------------------------------------------------------
%
%
\section{A perturbative construction of isospectral graphs}
\label{sec:examples}
%
%-----------------------------------------------------------------------

We base our discrete perturbative approach (spectral bracketing
technique) on the result presented in \Thm{main.idea} and using
auxiliary graphs having suitable symmetry.
Note that in the following classes of examples of isospectral graphs (see
e.g. \Thm{fuzzy-ball}) we can determine the spectrum explicitly due to
the fact that we know the spectra of the auxiliary graphs.
But our technique only requires to know that
the multiplicity of certain eigenvalues of the auxiliary graphs is
high enough to guarantee isospectrality. Our bracketing method is, in this sense, based on
\emph{multiplicities} of certain eigenvalues and not on explicit
computation of the spectrum or eigenfunctions of the auxiliary
graphs. In \cite{fclp:22c} we develop a new geometrical construction
of families of isospectral magnetic graphs.  This method starts with
several copies of an arbitrary building block graph $G$ which are
glued following a specific pattern.  To prove isospectrality in the
mentioned reference we lift eigenfunctions on $G$ and, also,
eigenfunctions with certain Dirichlet conditions on the merged
vertices to the constructed graph. We mention that the perturbative
method developed in this article can also be applied to alternatively
prove isospectrality of some families of graphs constructed in
\cite{fclp:22c}.

Our construction of non-isomorphic graphs in some examples is based on
the combinatorial notion of partition of length $s\in \N$ of a given
natural number $r\in\N$ (which we call an $s$-partition of $r$ for
short). Recall that an $s$ partition of $r$ is given by a multiset
$A=\{a_1,a_2,\dots,a_s\}$ (i.e. taking multiplicities into account)
with $a_i\in\N$ and $a_1+\dots+a_s=r$. We will show that different
$s$-partitions of $r$ lead to isospectral non-isomorphic graphs. As an
example, the smallest natural number having two different partitions
is $r=4$: the length is $s=2$ and the partitions are $\{1,3\}$ and
$\{2,2\}$ (since $4=2+2=1+3$).

%-----------------------------------------------------------------------
\subsection{The fuzzy ball construction}
\label{subsec:fuzzy-ball}
%-----------------------------------------------------------------------

We revisit and generalise the fuzzy ball construction
of~\cite[Section~4]{butler-grout:11}, proof isospectrality using
spectral bracketing technique and \emph{without} computing
eigenfunctions.  Let $K_r$ be the complete graph with $r$ vertices
and denote by $\wh {K}_r$ the graph with $2r$ vertices obtained
from $K_r$ by attaching a pendant edge to each vertex of
$K_r$.

We first calculate the eigenvalues with high multiplicity since we
only need this information to prove isospectrality of the family of
graphs in \Thm{fuzzy-ball}.
%-----------------------------------------------------------------------
\begin{lemma}[Spectrum of auxiliary graphs]
  \label{lem:lambda.mult} \indent
  \begin{enumerate}
  \item
    \label{lambda.mult.c}
    If $K_r$ is the complete graph with $r$ vertices, then
    \begin{equation*}
      1 \pm \frac 1 {r-1}
      \in \spec{K_r^\mp}
    \end{equation*}
    has multiplicity $r-1$ in the spectrum of the (signless) standard
    Laplacian (here and in the sequel, the upper sign is for the
    standard Laplacian, the lower sign for the signless Laplacian).
  \item
    \label{lambda.mult.d}
    Let $\wh{K}_r$ be the complete graph with a pendant edge
    decoration at each vertex of $K_r$. Then
    \begin{equation*}
      \frac{2r\pm 1-\sqrt{4r+1}}{2r},
      \frac{2r\pm 1+\sqrt{4r+1}}{2r}
      \in \spec {\wh {K}_r^\mp}
    \end{equation*}
    with multiplicity $r-1$.
  \end{enumerate}
\end{lemma}
%-----------------------------------------------------------------------
\begin{proof}
  Part~\itemref{lambda.mult.c} is standard. To show
  part~\itemref{lambda.mult.d} we use a factorisation result of
  characteristic polynomials in \cite{heydari:19}.  In fact, the
  decorated complete graph $\wh{K}_r$ can be seen as a rooted product
  of ${K}_r$ with edges $K_2$ on each vertex. Since ${K}_r$ a is
  simple $(r-1)$-regular graph and since ${K}_2$ has degree $1$ we
  obtain from Corollary~1 of~\cite{heydari:19} after some
  straightforward manipulations that the spectrum is as claimed
  in~\eqref{lambda.mult.d}.
\end{proof}
%-----------------------------------------------------------------------
\begin{figure}[h]
  \centering {
    \begin{tikzpicture}[auto, vertex/.style={circle,draw=black!100,%
        fill=black!100, %thick,
        inner sep=0pt,minimum size=1mm},scale=1] %
      \foreach \x in {0,60,120,180,240,300,360} %
      { %
        \node (C\x) at ({cos(\x)},{sin(\x)}) [vertex,%fill=red,
        label=below:] {}; %
        \node (B\x) at ({2*cos(\x)},{2*sin(\x)}) [vertex,label=below:] {};%
        \path [-](C\x) edge node[right] {} (B\x);%
      }%
      \foreach \x in {0,60,120,180,240,300,360} %
      { %
        \foreach \y in {0,60, ..., \x}
        { %
          \path [-](C\x) edge node[right] {} (C\y);
        } %
      }%
      \draw[] (0.35,-1.5)node[left] {\small$\wh{K}_6$};
    \end{tikzpicture} %
    \quad \quad \quad
    \begin{tikzpicture}[auto, vertex/.style={circle,draw=black!100,%
        fill=black!100, %thick,
        inner sep=0pt,minimum size=1mm},scale=1] %
      \foreach \x in {0,60,120,180,240,300,360} %
      { %
        \node (C\x) at ({cos(\x)},{sin(\x)}) [vertex,%fill=red,
        label=below:] {}; %
        \foreach \y in {0,60, ..., \x}
        { %
          \path [-](C\x) edge node[right] {} (C\y);
          % \node (A\x) at ({0*cos(\x)},{0*sin(\x)}) [vertex,label=below:] {};
          % \path [-](B\x) edge node[right] {} (A\x);
        } %
      }%
      \node (B1) at (0,2)[vertex,label=below:] {}; %
      \path [-] (B1) edge node[right]{} (C60); %
      \path [-] (B1) edge node[right]{} (C120); %
      \node (B2) at (-1.41,-0.88)[vertex,label=below:] {}; %
      % John, maybe you can do it better ...? If I put -1 instead of
      % -0.88 it looks strange ...?
      \path [-] (B2) edge node[right]{} (C180); %
      \path [-] (B2) edge node[right]{} (C240); %
      \node (B3) at (1.41,-0.88)[vertex,label=below:] {}; %
      \path [-] (B3) edge node[right]{} (C300); %
      \path [-] (B3) edge node[right]{} (C360); %
      %
      % \foreach \x in {60,180,300} %
      % { %
      %   \node (B\x) at ({2*cos(\x+30)},{2*sin(\x+30)}) %
      %   [vertex,label=below:] {}; %
      %   \path [-] (C\x) edge node[right]{} (B\x);
      %   \path [-] (C{\x+30}) edge node[right]{} (B\x);
      % }
      \draw[] (0.35,-1.5)node[left] {\small$\wh{K}_6(A)$};
    \end{tikzpicture}%
    \quad \quad \quad
    \begin{tikzpicture}[auto, vertex/.style={circle,draw=black!100,%
        fill=black!100, %thick,
        inner sep=0pt,minimum size=1mm},scale=1] %
      \foreach \x in {0,60,120,180,240,300,360} %
      { %
        \node (C\x) at ({cos(\x)},{sin(\x)}) [vertex,%fill=red,
        label=below:] {}; %
        \foreach \y in {0,60, ..., \x}
        { %
          \path [-](C\x) edge node[right] {} (C\y);
          % \node (A\x) at ({0*cos(\x)},{0*sin(\x)}) [vertex,label=below:] {};
          % \path [-](B\x) edge node[right] {} (A\x);
        } %
      }%
      \node (B1) at (2,0)[vertex,label=below:] {}; %
      \path [-] (B1) edge node[right]{} (C360); %
      \node (B2) at (0,2)[vertex,label=below:] {}; %
      % John, maybe you can do it better ...? If I put -1 instead of
      % -0.88 it looks strange ...?
      \path [-] (B2) edge node[right]{} (C60); %
      \path [-] (B2) edge node[right]{} (C120); %
      \node (B3) at (-1,-1.2)[vertex,label=below:] {}; %
      \path [-] (B3) edge node[right]{} (C180); %
      \path [-] (B3) edge node[right]{} (C240); %
      \path [-] (B3) edge node[right]{} (C300); %
      \draw[] (0.35,-1.5)node[left] {\small$\wh{K}_6(B)$};
    \end{tikzpicture}%
    \quad \quad \quad
    \begin{tikzpicture}[auto, vertex/.style={circle,draw=black!100,%
        fill=black!100, %thick,
        inner sep=0pt,minimum size=1mm},scale=1] %
      \foreach \x in {0,1,2,3,4,5,6,7} %
      { %
        \node (C\x) at ({cos(\x*360/7-12.9)},{sin(\x*360/7-12.9)})%
        [vertex,%fill=red,
        label=below:] {}; %
      }%
      \foreach \x in {0,1,2,3,4,5,6,7} %
      { %
        \foreach \y in {0, ..., \x}
        { %
          \path [-](C\x) edge node[right] {} (C\y);
        } %
      }%
      \draw[] (0.35,-1.5)node[left] {\small$K_7$};
    \end{tikzpicture} %
  } \caption{The figures above correspond to the auxiliary graph
    $\wh{K}_6$, the isospectral graphs $\wh{K}_6(A)$ and
    $\wh{K}_6(B)$ for $A=(2,2,2)$ and $B=(1,2,3)$, and finally the
    auxiliary graph $K_7$.  Note that $\wh{K}_6(A)$ and
    $\wh{K}_6(B)$ are obtained from $\wh{K}_6$ by contracting
    the pendant vertices according to $A$ and $B$, and that $K_7$
    is obtained from either $\wh{K}_6(A)$ or $\wh{K}_6(B)$ by
    contracting all formerly pendant vertices into one vertex.}
  \label{fig:fuzzy-ball.ex}
\end{figure}
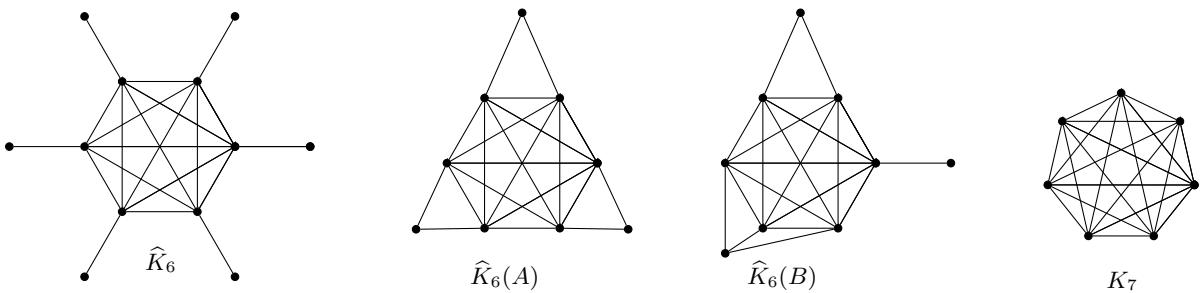
   % -------------------------------------------------------------------

We now construct the first example class (the so-called \emph{fuzzy
  balls}). Note that we can also treat the signless version without
any modification of the proof and that the signless spectrum cannot be
calculated from the unsigned one as the graphs are not regular. We
start with the decorated complete graph $\wh{K}_r$.  Let
$A=\lMset a_1,\dots,a_s \rMset$ be an $s$-partition of $r \in
\N$ and label the pendant vertices as
$\{v_1,\dots,v_r\}$. Consider the equivalence relation
$\sim_A$ which contracts the set of vertices
$\{v_1,\dots,v_{a_1}\}$, $\{v_{a_1+1},\dots,v_{a_1+a_2}\}$,
$\dots$, $\{v_{a_1+\dots+a_{n-1}+1},\dots,a_r\}$ into
$s$ vertices and denote the corresponding quotient graph by
$\wh {K}_r(A)=\quotient[A]{\wh{Gr
    K}_r}$.  In order to have two different $s$-partitions of
$r$ we necessarily need $r \ge 4$ and $s \in \{2,\dots, r-1\}$.

%-----------------------------------------------------------------------
\begin{theorem}[``The fuzzy ball theorem'']
  \label{thm:fuzzy-ball}
  Assume that $A$ and $B$ are two different $s$-partitions of
  $r$. Then the graphs $\wh{K}_r(A)$ and $\wh{K}_r(B)$
  constructed above are non-isomorphic and
  isospectral for the standard and signless standard Laplacian.
\end{theorem}
%-----------------------------------------------------------------------
\begin{proof}
  We apply \Thm{main.idea} and show that the spectrum of
  $G :=\wh{K}_r(A)$ depends \emph{only} on $r$ and the length $s$
  of the partition. We choose as first bracketing graph the decorated
  complete graph $G_1=\wh{K}_r$ and as spectral set
  \begin{equation*}
    \Lambda_1
    =\Bigl\lMset
      \frac{(2r \pm 1-\sqrt{4r+1})}{2r},
      \frac{(2r \pm 1+\sqrt{4r+1})}{2r}
    \Bigr\rMset.
  \end{equation*}
  From \Lemenum{lambda.mult}{lambda.mult.d} we obtain
  $\Lambda_1^{(r-1)}\subset \spec{G_1^\mp}$ and hence,
  $\mu_1=r-1$. Since $\wh{K}_r(A)=\wh{K}_r/{\sim}_A$ has
  shrinking number $t_1:=r-s$ we obtain from \Prp{homomorphism} that
  \begin{equation*}
    \wh{K}_r^\mp\less\wh{K}_r(A)^\mp \less[r-s]\wh{K}_r^\mp.
  \end{equation*}

  For the second bracketing graph we consider $G_2=K_{r+1}$ and
  the spectral set $\Lambda_2:=\{1 \pm \nicefrac1r\}$ which has
  multiplicity $\mu_2:=r$ in $\spec{G_2^\mp}$.  Note that $K_{r+1}$
  can be obtained from $\wh{K}_r(A)$ by contracting the $s$
  vertices obtained by the contraction of the pendant vertices of
  $\wh{K}_r$ given by the relation $\sim_A$. The shrinking number
  of this process is $t_2=s-1$ and hence
  $\wh{K}_r(A)^\mp \less K_{r+1}^\mp \less[s-1] \wh{Gr
    K}_r(A)^\mp$, or, equivalently,
  \begin{equation*}
    K_{r+1}^\mp\less[s-1] \wh{K}_r(A)^\mp \less K_{r+1}^\mp.
  \end{equation*}

  It is clear that the spectral subsets $\Lambda_1$ and $\Lambda_2$
  are disjoint.  Moreover, we have already detected
  \begin{equation*}
    (\mu_1-t_1)\abs{\Lambda_1}
    + (\mu_2-t_2)\abs{\Lambda_2}
    =\underbrace{((r-1)-(r-s))}_{=s-1}\cdot 2
    + \underbrace{(r-(s-1))}_{r-s+1} \cdot 1
    = r+s -1
    = \abs{\wh{K}_r(A)}-1
  \end{equation*}
  eigenvalues, hence the spectrum of the Laplacian
  $\lapl{\wt {K}_r(A)}$ is determined by \Thm{main.idea} and
  \Remenum{main.idea}{one.ev.missing}.

  Since the values in $\Lambda_1$ and $\Lambda_2$ as well as the
  multiplicities depend only on $r$ and the length $s$ and not on the
  concrete partition $A$ or $B$ we obtain
  $\spec{\wh{K}_r(A)^\mp}=\spec{\wh{K}_r(B)^\mp}$ as claimed.

  Finally, the degree lists of $\wh{K}_r(A)$ and $\wh{K}_r(B)$
  are given respectively by
  \begin{equation*}
    \lMset r^{(r)} \rMset \uplus A \quadtext{and}
    \lMset r^{(r)} \rMset \uplus B,
  \end{equation*}
  where the union of multisets takes multiplicities into
  account. Since the partitions $A$ and $B$ are different we conclude
  that the corresponding graphs cannot be isomorphic.
\end{proof}
%-----------------------------------------------------------------------

%-----------------------------------------------------------------------
\begin{corollary}
  \label{cor:fuzzy-ball}
  The spectrum of $\wh{K}_r(A)^\mp$ and $\wh{K}_r(B)^\mp$ is given by
  \begin{equation*}
%    \label{eq:fuzzy-spec}
    \spec{\wh{K}_r(A)^\mp}
    =  \spec{\wh{K}_r(B)^\mp}
    = \Bigl\lMset
    \frac{2r \pm 1-\sqrt{4r+1}}{2r},
    \frac{2r \pm 1+\sqrt{4r+1}}{2r}
    \Bigr\rMset^{(s-1)}
    \cup
    \Bigl\lMset
    1 \pm \frac1r
    \Bigr\rMset^{(r-s+1)}
    \cup \lMset \lambda_*^\pm \rMset,
  \end{equation*}
  where powers $(\cdot)^{(\mu)}$ indicate multiplicity.  The remaining
  eigenvalue is $\lambda_*^\pm=1\pm 1$, i.e., $0$ for the standard and
  $2$ for the signless Laplacian.
\end{corollary}
%-----------------------------------------------------------------------

For the next result we need to mention the notion of metric graph.
Recall that any discrete graph $G$ has an associated equilateral
metric graph which we denote here by $\overline{G}$ (see
\cite{lledo-post:08b,berko-kuch:13,kurasov-muller:21} and references
cited therein).  Recall that the standard Laplacian (sometimes also
called ``Kirchhoff'') is an unbounded second order differential
operator on the metric graph with standard conditions on the vertices
(the sum of the in-derivatives equals the sum of the out-derivatives)
which guarantee that the operator is self-adjoint.

We observe that the construction of isospectral graphs as in
\Thm{fuzzy-ball} produces also isospectral equilateral metric graphs
for the corresponding Kirchhoff Laplacians.  The proof is based on a
beautiful relation between the spectra of the standard Laplacian and
the Kirchhoff Laplacian (see, e.g.~\cite{berko-kuch:13,lledo-post:08b}
and references therein).
%-----------------------------------------------------------------------
\begin{corollary}
  \label{cor:metric-1}
  Let $A$ and $B$ are two different $s$-partitions of $r$ and denote
  by $\overline{K}_r(A)$ and $\overline{K}_r(B)$ the
  equilateral metric graph associated with the discrete graphs
  $\wh{K}_r(A)$ and $\wh{K}_r(B)$.  Then the graphs
  $\overline{K}_r(A)$ and $\overline{K}_r(B)$ are
  non-isomorphic and isospectral for the standard metric Laplacian.
\end{corollary}
%-----------------------------------------------------------------------
\begin{proof}[Sketch of the proof]
  We prove in~\cite{fclp:22c} that $\overline G_1$ and
  $\overline G_2$ are isospectral with respect to the standard metric
  Laplacian if and only if $G_1$ and $G_2$ are isospectral with
  respect to the standard discrete Laplacian and if $G_1$ and $G_2$
  have the same number of edges.  A similar statement is stated
  in~\cite[Proposition~1]{kurasov-muller:21}. (Note that isospectrality
  guarantees that the number of vertices is the same and, therefore, preservation
  of the number of edges guarantees that the corresponding Betti numbers are also the same.)
  Note that this follows for eigenvalues $\lambda \in (0,2)$ by e.g.~\cite[Proposition~4.1, 4.7
  and~5.2]{lledo-post:08b}.  Recall that the bipartiteness of a graph
  can be seen from the spectrum of the standard discrete Laplacian (a
  connected graph is bipartite if and only if $2$ is in its spectrum).

  The claim on isospectral metric graphs follows from the fact that in
  all our constructions of isospectral graphs using vertex
  contraction, the number of edges is the same.
\end{proof}
%-----------------------------------------------------------------------

%-----------------------------------------------------------------------
\begin{remark}
  Note that the splitted complete graphs $K_5$ in Figure~1 of
  \cite{kurasov-muller:21} fall into the class of fuzzy balls. In
  fact, they correspond to $\wh{K}_4(A)$ and $\wh{K}_4(B)$
  with partitions of $r=4$ given by $A=(2,2)$ and $B=(1,3)$. Therefore
  the corresponding equilateral metric graphs $\overline{K}_4(A)$
  and $\overline{K}_4(B)$ will also be isospectral for the
  corresponding Kirchhoff Laplacian.
\end{remark}
%-----------------------------------------------------------------------

%------------------------------------------------------------------
\begin{example}
  We illustrate here in a concrete example with $r=6$ and $s=3$
  (cf. \Fig{fuzzy-ball.ex}) how to apply directly and with the help of
  diagrams the spectral bracketing technique developed in the proof
  \Thm{fuzzy-ball}.  Recall that the lowest eigenvalue $0$ is always
  included in all spectra.

  The first bracketing here arises from
  \begin{equation*}
    \wh{K}_6\less \wh{K}_6(A) \less[3]\wh{K}_6.
  \end{equation*}
  with shrinking number $t_1:=r-s=3$. Hence, for the eigenvalues we have
  \vspace{2mm}
  \begin{center}
    \begin{tabular}{lp{3ex}p{3ex}p{3ex}%
      |p{3ex}|>{\columncolor[gray]{0.75}}p{3ex}%
      |>{\columncolor[gray]{0.75}}p{3ex}%
      |p{3ex}|p{3ex}|p{3ex}%
      |p{3ex}|>{\columncolor[gray]{0.75}}p{3ex}%
      |>{\columncolor[gray]{0.75}}p{3ex}%
      |p{3ex}p{3ex}p{3ex}}
      \hhline{~~~~------------}
      $\wh{K}_6$& & &
      & 0 & $\nicefrac23$ & $\nicefrac23$ %
      & $\nicefrac23$ & $\nicefrac23$ &$\nicefrac23$
      & $*$ & $\nicefrac32$ &$\nicefrac32$
      & \multicolumn{1}{p{3ex}|}{$\nicefrac32$}
      & \multicolumn{1}{p{3ex}|}{$\nicefrac32$}
      & \multicolumn{1}{p{3ex}|}{$\nicefrac32$}
      \\\hhline{~~~~------------}
      $\wh{K}_6(A)$ &&&  %
      & 0 & $\nicefrac23$ & $\nicefrac23$ %
      & ? & ? & ?
      & ? & $\nicefrac32$ & $\nicefrac32$
      \\\hhline{~------------}
      $\wh{K}_6$
      & \multicolumn{1}{|p{3ex}}{0}
      & \multicolumn{1}{|p{3ex}|}{$\nicefrac23$}
      & $\nicefrac23$
      & $\nicefrac23$ & $\nicefrac23$ & $\nicefrac23$
      & $*$ & $\nicefrac 32$ & $\nicefrac 32$ %
      & \multicolumn{1}{p{3ex}|}{$\nicefrac32$}
      & \multicolumn{1}{>{\columncolor[gray]{0.75}}p{3ex}|}{$\nicefrac32$}
      & \multicolumn{1}{>{\columncolor[gray]{0.75}}p{3ex}|}{$\nicefrac32$}
      \\\hhline{~------------}
    \end{tabular}
  \end{center}
   \vspace{2mm}
  where the simple eigenvalue $*=\nicefrac 76$ of
  $\wh {K}_6$ is useless for this bracketing and the question
  marks are the four eigenvalues that still need to be determined.
  Note that $\nicefrac 23$ and $\nicefrac 32$ are eigenvalues of
  $\wh {K}_6$ both with multiplicity $\mu_1-(r-s)=5-(6-3)=2$ and
  determine through enclosure four eigenvalues of $\wh{K}_6(A)$
  (light gray in the diagram).

  To determine the remaining eigenvalues we consider a second
  bracketing by choosing $G_2=K_7$ and the spectral set
  $\Lambda_2:=\{\nicefrac76\}$ which has multiplicity $\mu_2:=6$ in
  $\spec{G_2}$. The shrinking number is $t_2=2$ and hence
  \begin{equation*}
    K_7\less[2] \wh{K}_6(A) \less K_7.
  \end{equation*}
  These relations can be represented diagrammatically for the
  eigenvalues as
   \vspace{2mm}
  \begin{center}
    \begin{tabular}{lp{3ex}p{3ex}%
      |p{3ex}%
      |>{\columncolor[gray]{0.6}}p{3ex}%
      |>{\columncolor[gray]{0.6}}p{3ex}%
      |>{\columncolor[gray]{0.6}}p{3ex}%
      |>{\columncolor[gray]{0.6}}p{3ex}%
      |p{3ex}p{3ex}}
      \hhline{~~~-------}
      $K_7$ &&
      & $0$ & $\nicefrac76$ & $\nicefrac76$ %
      & $\nicefrac76$ & $\nicefrac76$%
      & \multicolumn{1}{p{3ex}|}{$\nicefrac76$}
      & \multicolumn{1}{p{3ex}|}{$\nicefrac76$}
      \\\hhline{~---------}
      $\wh{K}_6(A)$ %
      & \multicolumn{1}{|p{3ex}}{0}
      & \multicolumn{1}{|p{3ex}|}{$\nicefrac23$}
      & $\nicefrac23$ %
      & $\nicefrac76$ & $\nicefrac76$ & $\nicefrac76$
      & $\nicefrac76$
      & \multicolumn{1}{p{3ex}|}{$\nicefrac32$}
      & \multicolumn{1}{p{3ex}|}{$\nicefrac32$}
      \\\hhline{~---------}
      $K_7$
      & \multicolumn{1}{|p{3ex}}{$0$}
      & \multicolumn{1}{|p{3ex}|}{$\nicefrac76$}
      & $\nicefrac76$ & $\nicefrac76$ & $\nicefrac76$
      & $\nicefrac76$ & $\nicefrac 76$ %
      \\\hhline{~-------~~}
    \end{tabular}
  \end{center}
   \vspace{2mm}
  Therefore the remaining four eigenvalues are equal to
  $\nicefrac76$ since multiplicity is given by $\mu_2-(s-1)=6-(3-1)=4$
  (dark grey).  Altogether we have that both bracketings determine the
  complete spectrum of $\wh {K}_6(A)$:
   \vspace{2mm}
  \begin{center}
    \begin{tabular}{l%
      |>{\columncolor[gray]{1}}p{3ex}%
      |>{\columncolor[gray]{0.75}}p{3ex}%
      |>{\columncolor[gray]{0.75}}p{3ex}%
      |>{\columncolor[gray]{0.6}}p{3ex}%
      |>{\columncolor[gray]{0.6}}p{3ex}%
      |>{\columncolor[gray]{0.6}}p{3ex}%
      |>{\columncolor[gray]{0.6}}p{3ex}%
      |>{\columncolor[gray]{0.75}}p{3ex}%
      |>{\columncolor[gray]{0.75}}p{3ex}|}
      \hhline{~---------}
      & $0$ & $\nicefrac23$ & $\nicefrac23$ %
      & $\nicefrac76$ & $\nicefrac76$ & $\nicefrac76$
      & $\nicefrac76$ & $\nicefrac32$ & $\nicefrac32$
      \\\hhline{~---------}
    \end{tabular}
  \end{center}
   \vspace{2mm}
  Note that in the reasoning only $r=6$ and the length of the
  partition $s=3$ matter and so the two different partitions
  $A=\lMset 2,2,2\rMset$ and $B=\lMset 1,2,3 \rMset$ lead to
  isospectral graphs $\wh{K}_6(A)$ and $\wh{Gr
    K}_6(B)$. Moreover, the corresponding degree lists
  \begin{equation*}
    \deg \wh{K}_6(A)=(\mathbf 2,\mathbf2,\mathbf2,6,6,6,6,6,6)
    \qquadtext{and}
    \deg \wh{K}_6(B)=(\mathbf1,\mathbf2,\mathbf3,6,6,6,6,6,6)
  \end{equation*}
  are different and, hence, the graphs are not isomorphic.

  Finally, for the signless Laplacian one can reason similarly
  choosing the same auxiliary graphs and replacing the eigenvalues $0$
  by $2$ (constant eigenfunction), $\nicefrac23$ by $\nicefrac12$,
  $\nicefrac32$ by $\nicefrac43$ and $\nicefrac76$ by $\nicefrac56$,
  respectively.
\end{example}
%-----------------------------------------------------------------------

%-----------------------------------------------------------------------
\subsection{The fuzzy complete bipartite construction}
\label{subsec:complete-bipartite}
%-----------------------------------------------------------------------
In this section we will show a different class of examples leading to
isospectral graphs. Instead of following the lemma-theorem structure
of the preceding subsection we will directly reason using spectral
diagrams in concrete examples that can be generalised in an obvious
way.

For $p, r \in \N$, let $K_{p,r}$ be the complete bipartite graph
with $p+r$ vertices and denote by $\wh {K}_{p,r}$ the graph with
$p+2r$ vertices obtained from $K_{p,r}$ by attaching a pendant
edge to each of the $r$ vertices in the partition of $K_{p,r}$.
Note since the resulting classes of graphs in this subsection will be
bipartite there is no need to distinguish between Laplacian and
signless Laplacian as they have the same spectrum.  This is due to the
fact that on a bipartite graph, the signatures $1$ and $-1$ on each
edge are gauge-equivalent (see e.g.~\cite[Sec.~3.1]{fclp:18}.  We now
construct the isospectral graphs from an $s$-partition $A$ of the
natural number $r$, i.e., $A=(a_1,\dots,a_r)$.  We start with the
decorated complete bipartite graph $\wh{K}_{p,r}$ (decorated with
$r$ pendant edges numbered by $\{v_1,\dots,v_r\}$).  Let
$A=\lMset a_1,\dots,a_s \rMset$ be an $s$-partition of $r \in \N$.
Consider the equivalence $\sim_A$ which contracts the set of vertices
$\{v_1,\dots,v_{a_1}\}$, $\{v_{a_1+1},\dots,v_{a_1+a_2}\}$, $\dots$,
$\{v_{a_1+\dots+a_{n-1}+1},\dots,v_r\}$ into $s$ vertices and denote
the corresponding quotient graph by
$\wh {K}_{p,r}(A)=\quotient[A]{\wh{K}_{p,r}}$.

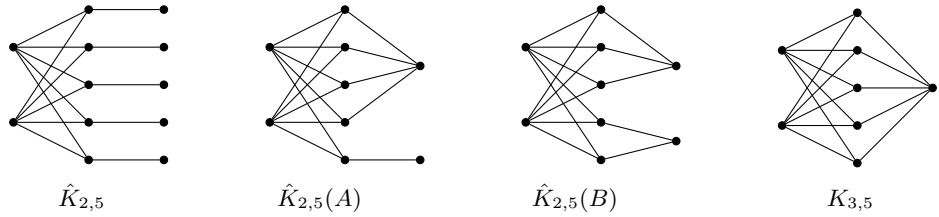
\begin{figure}[h]
  \centering {
    \begin{tikzpicture}[auto, vertex/.style={circle,draw=black!100,%
        fill=black!100, %thick,
        inner sep=0pt,minimum size=1mm},scale=1] %
      \foreach \r in {1,2,3,4,5} %
      { %
        \node (C\r) at (0,\r/2) [vertex,%fill=red,
        label=below:] {}; %
        \node (B\r) at (1,\r/2) [vertex,label=below:] {};%
        \path [-](C\r) edge node[right] {} (B\r);%
      }%

      \foreach \p in {2,4} %
      { %
        \node (D\p) at (-1,\p/2) [vertex,%fill=red,
        label=below:] {}; %
      }%

      \foreach \r in {1,2,3,4,5} %
      { %
        \foreach \p in {2,4} %
        { %
          \path [-](C\r) edge node[right] {} (D\p);%
        }%
      }%

      \draw[] (0.35,0)node[left] {\small$\hat{K}_{2,5}$};
    \end{tikzpicture} %
    %%%%%%%%%%%%%%%%%%%%%%%%%%%%%%%%%%%%% Part 2
    \quad \quad \quad
    \begin{tikzpicture}[auto, vertex/.style={circle,draw=black!100,%
        fill=black!100, %thick,
        inner sep=0pt,minimum size=1mm},scale=1] %
      \foreach \r in {1,2,3,4,5} %
      { %
        \node (C\r) at (0,\r/2) [vertex,%fill=red,
        label=below:] {}; %
      }%
      \foreach \s in {2,7} %
      { %
        \node (E\s) at (1,\s/4) [vertex,%fill=red,
        label=below:] {}; %
      }
      \foreach \r in {2,3,4,5}  {\path [-](C\r) edge node[right] {} (E7);}
      \foreach \r in {1}  {\path [-](C\r) edge node[right] {} (E2);}

      \foreach \p in {2,4} %
      { %
        \node (D\p) at (-1,\p/2) [vertex,%fill=red,
        label=below:] {}; %
      }%

      \foreach \r in {1,2,3,4,5} %
      { %
        \foreach \p in {2,4} %
        { %
          \path [-](C\r) edge node[right] {} (D\p);%
        }%
      }%

      \draw[] (0.35,0)node[left] {\small$\hat{K}_{2,5}(A)$};
    \end{tikzpicture}
    %%%%%%%%%%%%%%%%%%%%%%%%%%%%%%%%%%%%% Part 3
    \quad \quad \quad
    \begin{tikzpicture}[auto, vertex/.style={circle,draw=black!100,%
        fill=black!100, %thick,
        inner sep=0pt,minimum size=1mm},scale=1] %
      \foreach \r in {1,2,3,4,5} %
      { %
        \node (C\r) at (0,\r/2) [vertex,%fill=red,
        label=below:] {}; %
      }%
      \foreach \s in {3,7} %
      { %
        \node (E\s) at (1,\s/4) [vertex,%fill=red,
        label=below:] {}; %
      }
      \foreach \r in {3,4,5}  {\path [-](C\r) edge node[right] {} (E7);}
      \foreach \r in {1,2}  {\path [-](C\r) edge node[right] {} (E3);}

      \foreach \p in {2,4} %
      { %
        \node (D\p) at (-1,\p/2) [vertex,%fill=red,
        label=below:] {}; %
      }%

      \foreach \r in {1,2,3,4,5} %
      { %
        \foreach \p in {2,4} %
        { %
          \path [-](C\r) edge node[right] {} (D\p);%
        }%
      }%

      \draw[] (0.35,0)node[left] {\small$\hat{K}_{2,5}(B)$};
    \end{tikzpicture}
    %%%%%%%%%%%%%%%%%%%%%%%%%%%%%%%%%%%%% Part 4
    \quad \quad \quad
    \begin{tikzpicture}[auto, vertex/.style={circle,draw=black!100,%
        fill=black!100, %thick,
        inner sep=0pt,minimum size=1mm},scale=1] %
      \foreach \r in {1,2,3,4,5} %
      { %
        \node (C\r) at (0,\r/2) [vertex,%fill=red,
        label=below:] {}; %
      }%
      \foreach \s in {6} %
      { %
        \node (E\s) at (1,\s/4) [vertex,%fill=red,
        label=below:] {}; %
      }
      \foreach \r in {1,2,3,4,5}  {\path [-](C\r) edge node[right] %
        {} (E6);}
      \foreach \p in {2,4} %
      { %
        \node (D\p) at (-1,\p/2) [vertex,%fill=red,
        label=below:] {}; %
      }%

      \foreach \r in {1,2,3,4,5} %
      { %
        \foreach \p in {2,4} %
        { %
          \path [-](C\r) edge node[right] {} (D\p);%
        }%
      }%

      \draw[] (0.35,0)node[left] {\small${K}_{3,5}$};
    \end{tikzpicture}

  } \caption{The figures above represent the auxiliary graph $\wh{K}_{2,5}$, the isospectral
    graphs $\wh{K}_{2,5}(A)$ and $\wh{K}_{2,5}(B)$ for
    $A=(4,1)$ and $B=(2,3)$, and, finally, the auxiliary graph
    ${K}_{3,5}$.  Note that $\wh{K}_{2,5}(A)$ and
    $\wh{K}_{2,5}(B)$ are obtained from $\wh{K}_{2,5}$ by
    contracting the pendant vertices according to $A$ and $B$, and
    that ${K}_{3,5}$ is obtained from either $\wh{K}_{2,5}(A)$
    or $\wh{K}_{2,5}(B)$ by contracting all former pendant
    vertices once again into one vertex.}
  \label{fig:complete-bipartite.ex}
\end{figure}

To illustrate the bracketing technique directly in terms of diagrams
we will consider a concrete example of this class by choosing the case
$p=2$, $r=5$ and $s=2$ (see \Fig{complete-bipartite.ex}). The method
works similarly for other cases of in this class.

For the first bracketing we use $G_1=\wh{K}_{2,5}$ as auxiliary
graph which, by construction, give the relations
\begin{equation*}
  \wh{K}_{2,5} \less\wh{K}_{2,5}(A) \less[3] \wh{K}_{2,5}.
\end{equation*}
with shrinking number $t_1:=r-s=3$. Diagrammatically, the eigenvalues satisfy
 \vspace{2mm}
\begin{center}
  \begin{tabular}{lp{3ex}p{3ex}p{3ex}%
    |p{3ex}|>{\columncolor[gray]{0.75}}p{3ex}%
    |%>{\columncolor[gray]{0.75}}%
    p{3ex}%
    |p{3ex}|p{3ex}|p{3ex}%
    |p{3ex}|>{\columncolor[gray]{0.75}}p{3ex}%
    |%>{\columncolor[gray]{0.75}}%
    p{3ex}%
    |p{3ex}p{3ex}p{3ex}}
    \hhline{~~~~------------}
    $\wh{K}_{2,5}$& & &
    & $0$ & $\nicefrac12$ & $\nicefrac12$ %
    & $\nicefrac12$ & $\nicefrac12$ &$*$
    & $*$ & $\nicefrac32$ &$\nicefrac32$
    & \multicolumn{1}{p{3ex}|}{$\nicefrac32$}
    & \multicolumn{1}{p{3ex}|}{$\nicefrac32$}
    & \multicolumn{1}{p{3ex}|}{$2$}
    \\\hhline{~~~~------------}
    $\wh{K}_{2,5}(A)$ &&&  %
      & 0 & $\nicefrac12$ & ? %
      & ? & ? & ?
      & ? & $\nicefrac32$ & 2
      \\\hhline{~------------}
      $\wh{K}_{2,5}$
      & \multicolumn{1}{|p{3ex}}{$0$}
      & \multicolumn{1}{|p{3ex}|}{$\nicefrac12$}
      & $\nicefrac12$
      & $\nicefrac12$ & $\nicefrac12$ & $*$
      & $*$ & $\nicefrac32$ & $\nicefrac 32$ %
      & \multicolumn{1}{p{3ex}|}{$\nicefrac32$}
      & \multicolumn{1}{>{\columncolor[gray]{0.75}}p{3ex}|}{$\nicefrac32$}
      & \multicolumn{1}{%>{\columncolor[gray]{0.75}}%
        p{3ex}|}{$2$}
    \\\hhline{~------------}
  \end{tabular}
\end{center}
 \vspace{2mm}
where now the double eigenvalue $*=1$ is irrelevant for this
bracketing.  Note that we are able to enclose the two eigenvalues
$\nicefrac13$ and $\nicefrac32$, both with multiplicity
$\mu_1-(r-s)=4-(6-3)=1$ (light grey in the diagram).

To determine the remaining eigenvalues we consider a second bracketing
with the auxiliary graph $G_2=K_{3,5}$. Note that the eigenvalue
$1$ has multiplicity $\mu_2:=(p+1)+r-2=6$ in $\spec{G_2}$ and the
shrinking number is $t_2=s-1=1$.  Therefore we obtain the spectral
relations
\begin{equation*}
  K_{3,5} \less[1] \wh{K}_{2,5}(A) \less K_{3,5}.
\end{equation*}
which diagrammatically can be represented by
 \vspace{2mm}
\begin{center}
  \begin{tabular}{lp{3ex}
      |p{3ex}%
      |>{\columncolor[gray]{0.6}}p{3ex}%
      |>{\columncolor[gray]{0.6}}p{3ex}%
      |>{\columncolor[gray]{0.6}}p{3ex}%
      |>{\columncolor[gray]{0.6}}p{3ex}%
      |>{\columncolor[gray]{0.6}}p{3ex}%
      |p{3ex}p{3ex}}
      \hhline{~~--------}
      $K_{3,5}$ &&  $0$
      & $1$ & $1$ & $1$ %
      & $1$ & $1$%
      & \multicolumn{1}{p{3ex}|}{$1$}
      & \multicolumn{1}{p{3ex}|}{$2$}
      \\\hhline{~---------}
      $\wh{K}_{2,5}(A)$ %
      & \multicolumn{1}{|p{3ex}}{0}
      & \multicolumn{1}{|p{3ex}|}{$\nicefrac12$}
      & $1$ %
      & $1$ & $1$ & $1$
      & $1$
      & \multicolumn{1}{p{3ex}|}{$\nicefrac32$}
      & \multicolumn{1}{p{3ex}|}{2}
      \\\hhline{~---------}
      $K_{3,5}$
      & \multicolumn{1}{|p{3ex}}{$0$}
      & \multicolumn{1}{|p{3ex}|}{$1$}
      & $1$ & $1$ & $1$
      & $1$ & $1$
      & \multicolumn{1}{|p{3ex}|}{$2$}%
      \\\hhline{~--------~}
  \end{tabular}
\end{center}
 \vspace{2mm}
We have now fixed the eigenvalue $1$ with multiplicity
$\mu_2-t_2=6-1=5$ (dark grey in the diagram).

Altogether, the spectrum of $\wh {K}_{2,5}(A)$ is given by
 \vspace{2mm}
\begin{center}
    \begin{tabular}{l%
      |>{\columncolor[gray]{1}}p{3ex}%
      |>{\columncolor[gray]{0.75}}p{3ex}%
      |>{\columncolor[gray]{0.75}}p{3ex}%
      |>{\columncolor[gray]{0.6}}p{3ex}%
      |>{\columncolor[gray]{0.6}}p{3ex}%
      |>{\columncolor[gray]{0.6}}p{3ex}%
      |>{\columncolor[gray]{0.6}}p{3ex}%
      |>{\columncolor[gray]{0.75}}p{3ex}%
      |>{\columncolor[gray]{1}}p{3ex}|}
      \hhline{~---------}
      & $0$ & $\nicefrac12$ & $1$ %
      & $1$ & $1$ & $1$
      & $1$ & $\nicefrac32$ & $2$
      \\\hhline{~---------}
  \end{tabular}
\end{center}
 \vspace{2mm}
Again, the reasoning only depends on the values $p=2$, $r=6$ and the
length of the partition $s=2$ and so the two different partitions
$A=\{2,3\}$ and $B=\{1,4\}$ lead to isospectral graphs
$\wh{K}_{2,5}(A)$ and $\wh{K}_{2,5}(B)$. Since their degree
lists
\begin{equation*}
  \deg \wh{K}_{2,5}(A)=(\mathbf2,\mathbf3,3,3,3,3,3,5,5)
  \qquadtext{and}
  \deg \wh{K}_{2,5}(B)=(\mathbf1,3,3,3,3,3,\mathbf4,5,5)
\end{equation*}
are different the graphs are not isomorphic. To make contact with the
main result in \Thm{main.idea} recall that in the two bracketings used
in this class of examples we have for the first auxiliary graph
$G_1=\wh{K}_{2,5}$ the spectral set
$\Lambda_1=\left\{\nicefrac12,\nicefrac32 \right\}$ with
multiplicity $\mu_1=5-1=4$ and for the second auxiliary graph
$G_2={K}_{3,5}$ the spectral set is $\Lambda_2=\left\{ 1 \right\}$
with $\mu_2= 7 -1= 6$.
%-----------------------------------------------------------------------

As in \Cor{metric-1} we can extend isospectrality to the metric graph
scenario. Let $A$ and $B$ are two different $s$-partitions of $r$ and
denoting by $\overline{K}_{p,r}(A)$ and
$\overline{K}_{p,r}(B)$ the equilateral metric graph associated
with the discrete graphs $\wh{K}_{r,p}(A)$ and
$\wh{K}_{p,r}(B)$.  Then the graphs $\overline{K}_{p,r}(A)$
and $\overline{K}_{p,r}(B)$ are non-isomorphic and isospectral for
the Kirchhoff Laplacian.

%-----------------------------------------------------------------------
\subsection{Edge subdivision}
\label{subsec:example2}
%-----------------------------------------------------------------------

For presenting the following class of examples we first need to recall
an important operation on a graph $G=(V,E)$. The edge subdivision of
$e=\{u,v\}\in E$ consists in the deletion of $e=\{u,v\}$ from $G$, the
addition of a new vertex $w$ and two new edges $e_1=\{u,w\}$ and
$e_2=\{w,v\}$.  We denote by $S(G)$ the graph where all edges of $G$
are subdivided. Let $\wh{K}_5$ be the complete graph with a
pendent edge decoration at each vertex of ${K}_5$ as in
\Fig{fuzzy-ball.ex} and consider the corresponding edge subdivision
graph shown in \Fig{final.ex}. Since all graphs involved are bipartite
and connected we will always have $0$ and $2$ as simple eigenvalues in
the corresponding spectra.

%-----------------------------------------------------------------------
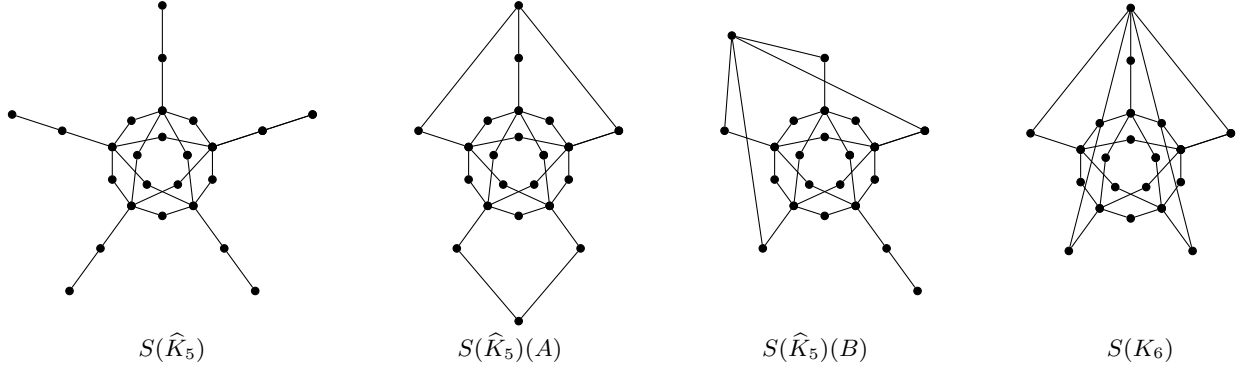
\begin{figure}[h]
  \centering {
    \begin{tikzpicture}[auto, vertex/.style={circle,draw=black!100,%
        fill=black!100, %thick,
        inner sep=0pt,minimum size=1mm},scale=.7] %
      \foreach \x in {18,54,...,378} %
      { %
        \node (A\x) at ({cos(\x)},{sin(\x)}) [vertex,%fill=red,
        label=below:] {}; %

      }%

      \foreach \x in {18,90,...,378} %
      { %
        \node (A\x) at ({cos(\x)},{sin(\x)}) [vertex,%fill=red,
        label=below:] {}; %
        \node (B\x) at ({2*cos(\x)},{2*sin(\x)}) [vertex,label=below:] {};%
        \node (C\x) at ({3*cos(\x)},{3*sin(\x)}) [vertex,label=below:] {};%
        \path [-](A\x) edge node[right] {} (B\x);%
        \path [-](B\x) edge node[right] {} (C\x);%

      }%

      \foreach \x in {18,90,...,378} %
      { %
        \node (E\x) at ({.5*cos(\x)},{.5*sin(\x)}) [vertex,%fill=red,
        label=below:] {}; %

      }%

      \path [-](A18) edge node[right] {} (A54);
      \path [-](A54) edge node[right] {} (A90);
      \path [-](A90) edge node[right] {} (A126);
      \path [-](A126) edge node[right] {} (A162);
      \path [-](A198) edge node[right] {} (A162);
      \path [-](A198) edge node[right] {} (A234);
      \path [-](A234) edge node[right] {} (A270);
      \path [-](A270) edge node[right] {} (A306);
      \path [-](A306) edge node[right] {} (A342);
      \path [-](A342) edge node[right] {} (A378);

      \path [-](A18) edge node[right] {} (E90);
      \path [-](A162) edge node[right] {} (E90);
      \path [-](A162) edge node[right] {} (E234);
      \path [-](A306) edge node[right] {} (E234);
      \path [-](A306) edge node[right] {} (E18);
      \path [-](A90) edge node[right] {} (E18);
      \path [-](A90) edge node[right] {} (E162);
      \path [-](A234) edge node[right] {} (E162);
      \path [-](A234) edge node[right] {} (E306);
      \path [-](E306) edge node[right] {} (A18);
      \draw[] (1,-3.5)node[left] {\small $S(\wh{K}_5)$};
    \end{tikzpicture} %
    \quad \quad \quad
    \begin{tikzpicture}[auto, vertex/.style={circle,draw=black!100,%
        fill=black!100, %thick,
        inner sep=0pt,minimum size=1mm},scale=.7] %
      \foreach \x in {18,54,...,378} %
      { %
        \node (A\x) at ({cos(\x)},{sin(\x)}) [vertex,%fill=red,
        label=below:] {}; %

      }%

      \foreach \x in {18,90,...,378} %
      { %
        \node (A\x) at ({cos(\x)},{sin(\x)}) [vertex,%fill=red,
        label=below:] {}; %
        \node (B\x) at ({2*cos(\x)},{2*sin(\x)}) [vertex,label=below:] {};%
        \path [-](A\x) edge node[right] {} (B\x);%

      }%

      \node (F1) at ({3*cos(90)},{3*sin(90)}) [vertex,%fill=red,
      label=below:] {}; %
      \path [-](F1) edge node[right] {} (B18);%
      \path [-](F1) edge node[right] {} (B90);%
      \path [-](F1) edge node[right] {} (B162);%

      \node (F2) at ({3*cos(270)},{3*sin(270)}) [vertex,%fill=red,
      label=below:] {}; %
      \path [-](F2) edge node[right] {} (B234);%
      \path [-](F2) edge node[right] {} (B306);%

      \foreach \x in {18,90,...,378} %
      { %
        \node (E\x) at ({.5*cos(\x)},{.5*sin(\x)}) [vertex,%fill=red,
        label=below:] {}; %

      }%

      \path [-](A18) edge node[right] {} (A54);
      \path [-](A54) edge node[right] {} (A90);
      \path [-](A90) edge node[right] {} (A126);
      \path [-](A126) edge node[right] {} (A162);
      \path [-](A198) edge node[right] {} (A162);
      \path [-](A198) edge node[right] {} (A234);
      \path [-](A234) edge node[right] {} (A270);
      \path [-](A270) edge node[right] {} (A306);
      \path [-](A306) edge node[right] {} (A342);
      \path [-](A342) edge node[right] {} (A378);

      \path [-](A18) edge node[right] {} (E90);
      \path [-](A162) edge node[right] {} (E90);
      \path [-](A162) edge node[right] {} (E234);
      \path [-](A306) edge node[right] {} (E234);
      \path [-](A306) edge node[right] {} (E18);
      \path [-](A90) edge node[right] {} (E18);
      \path [-](A90) edge node[right] {} (E162);
      \path [-](A234) edge node[right] {} (E162);
      \path [-](A234) edge node[right] {} (E306);
      \path [-](E306) edge node[right] {} (A18);
      \draw[] (1,-3.5)node[left] {\small $S(\wh{K}_5)(A)$};
    \end{tikzpicture} %
    \quad \quad \quad
    \begin{tikzpicture}[auto, vertex/.style={circle,draw=black!100,%
        fill=black!100, %thick,
        inner sep=0pt,minimum size=1mm},scale=.7] %
      \foreach \x in {18,54,...,378} %
      { %
        \node (A\x) at ({cos(\x)},{sin(\x)}) [vertex,%fill=red,
        label=below:] {}; %

      }%

      \foreach \x in {18,90,...,378} %
      { %
        \node (A\x) at ({cos(\x)},{sin(\x)}) [vertex,%fill=red,
        label=below:] {}; %
        \node (B\x) at ({2*cos(\x)},{2*sin(\x)}) [vertex,label=below:] {};%
        \path [-](A\x) edge node[right] {} (B\x);%

      }%

      \node (F1) at ({3*cos(126)},{3*sin(126)}) [vertex,%fill=red,
      label=below:] {}; %
      \path [-](F1) edge node[right] {} (B18);%
      \path [-](F1) edge node[right] {} (B90);%
      \path [-](F1) edge node[right] {} (B162);%
      \path [-](F1) edge node[right] {} (B234);%

      \node (F2) at ({3*cos(306)},{3*sin(306)}) [vertex,%fill=red,
      label=below:] {}; %

      \path [-](F2) edge node[right] {} (B306);%

      \foreach \x in {18,90,...,378} %
      { %
        \node (E\x) at ({.5*cos(\x)},{.5*sin(\x)}) [vertex,%fill=red,
        label=below:] {}; %

      }%

      \path [-](A18) edge node[right] {} (A54);
      \path [-](A54) edge node[right] {} (A90);
      \path [-](A90) edge node[right] {} (A126);
      \path [-](A126) edge node[right] {} (A162);
      \path [-](A198) edge node[right] {} (A162);
      \path [-](A198) edge node[right] {} (A234);
      \path [-](A234) edge node[right] {} (A270);
      \path [-](A270) edge node[right] {} (A306);
      \path [-](A306) edge node[right] {} (A342);
      \path [-](A342) edge node[right] {} (A378);

      \path [-](A18) edge node[right] {} (E90);
      \path [-](A162) edge node[right] {} (E90);
      \path [-](A162) edge node[right] {} (E234);
      \path [-](A306) edge node[right] {} (E234);
      \path [-](A306) edge node[right] {} (E18);
      \path [-](A90) edge node[right] {} (E18);
      \path [-](A90) edge node[right] {} (E162);
      \path [-](A234) edge node[right] {} (E162);
      \path [-](A234) edge node[right] {} (E306);
      \path [-](E306) edge node[right] {} (A18);
      \draw[] (1,-3.5)node[left] {\small $S(\wh{K}_5)(B)$};
    \end{tikzpicture} %
    \quad \quad \quad
    \begin{tikzpicture}[auto, vertex/.style={circle,draw=black!100,%
        fill=black!100, %thick,
        inner sep=0pt,minimum size=1mm},scale=.7] %
      \foreach \x in {18,54,...,378} %
      { %
        \node (A\x) at ({cos(\x)},{sin(\x)}) [vertex,%fill=red,
        label=below:] {}; %

      }%

      \foreach \x in {18,90,...,378} %
      { %
        \node (A\x) at ({cos(\x)},{sin(\x)}) [vertex,%fill=red,
        label=below:] {}; %
        \node (B\x) at ({2*cos(\x)},{2*sin(\x)}) [vertex,label=below:] {};%
        \path [-](A\x) edge node[right] {} (B\x);%

      }%

      \node (F1) at ({3*cos(90)},{3*sin(90)}) [vertex,%fill=red,
      label=below:] {}; %
      \path [-](F1) edge node[right] {} (B18);%
      \path [-](F1) edge node[right] {} (B90);%
      \path [-](F1) edge node[right] {} (B162);%
      \path [-](F1) edge node[right] {} (B234);%
      \path [-](F1) edge node[right] {} (B306);%

      \foreach \x in {18,90,...,378} %
      { %
        \node (E\x) at ({.5*cos(\x)},{.5*sin(\x)}) [vertex,%fill=red,
        label=below:] {}; %

      }%

      \path [-](A18) edge node[right] {} (A54);
      \path [-](A54) edge node[right] {} (A90);
      \path [-](A90) edge node[right] {} (A126);
      \path [-](A126) edge node[right] {} (A162);
      \path [-](A198) edge node[right] {} (A162);
      \path [-](A198) edge node[right] {} (A234);
      \path [-](A234) edge node[right] {} (A270);
      \path [-](A270) edge node[right] {} (A306);
      \path [-](A306) edge node[right] {} (A342);
      \path [-](A342) edge node[right] {} (A378);

      \path [-](A18) edge node[right] {} (E90);
      \path [-](A162) edge node[right] {} (E90);
      \path [-](A162) edge node[right] {} (E234);
      \path [-](A306) edge node[right] {} (E234);
      \path [-](A306) edge node[right] {} (E18);
      \path [-](A90) edge node[right] {} (E18);
      \path [-](A90) edge node[right] {} (E162);
      \path [-](A234) edge node[right] {} (E162);
      \path [-](A234) edge node[right] {} (E306);
      \path [-](E306) edge node[right] {} (A18);
      \draw[] (1,-3.5)node[left] {\small $S({K}_6)$};
    \end{tikzpicture} %
  } \caption{The figures above correspond to the auxiliary graph
    $S(\wh{K}_6)$, the isospectral graphs $S(\wh{K}_5)(A)$ and
    $S(\wh{K}_6)(B)$ for $A=(2,3)$ and $B=(1,4)$, and finally the
    auxiliary graph $S(K_6)$.  Note that $S(\wh{K}_6)(A)$ and
    $S(\wh{K}_6)(B)$ are obtained from $S(\wh{K}_5)$ by
    contracting the pendant vertices according to $A$ and $B$, and
    that $S(K_6)$ is obtained from either $S(\wh{K}_6)(A)$ or
    $S(\wh{K}_6)(B)$ by contracting all formerly pendant vertices
    once again into one vertex.  } \label{fig:final.ex}
\end{figure}
%-----------------------------------------------------------------------
We will show using spectral diagrams that $S(\wh{K}_5)(A)$ and
$S(\wh{K}_5)(B)$ are isospectral. As before we consider the first
graph and to determine the nontrivial $20$ eigenvalues of
$S(\wh{K}_5)(A)$ we begin with the auxiliary graph
$S(\wh{K}_5)$. By construction we have the spectral relations
\begin{equation*}
  S(\wh{K}_5)\less S(\wh{K}_5)(A) \less[3]S(\wh{K}_5),
\end{equation*}
with shrinking number $t_1:=r-s=3$ using \Prp{homomorphism}. Hence,
for the corresponding eigenvalues we have the following spectral
diagram

\begin{table}[h]
  \resizebox{\textwidth}{!}{%
    \begin{tabular}{llrrr|r|
      >{\columncolor[gray]{0.75}}r |r|r|r|r|
      >{\columncolor[gray]{0.75}}r |r|r|r|
      >{\columncolor[gray]{0.75}}r |
      >{\columncolor[gray]{0.75}}r |r|r|r|
      >{\columncolor[gray]{0.75}}r |r|r|r|r|
      >{\columncolor[gray]{0.75}}r |r|rrr}
%      \cline{6-30}
      \hhline{~~~~~-------------------------}
      $S(\wh{K}_5)$ &
      &
      &
      &
			&
			0 &
			$w_-$ &
			$w_-$ &
			$w_-$ &
			$w_-$ &
			$z_-$ &
			$\wh w_-$ &
			$\wh w_-$ &
			$\wh w_-$ &
			$\wh w_-$ &
			1 &
			1 &
			1 &
			1 &
			1 &
			$\wh w_+$ &
			$\wh w_+$ &
			$\wh w_+$ &
			$\wh w_+$ &
			$z_+$ &
			$w_+$ &
			$w_+$ &
			\multicolumn{1}{r|}{$w_+$} &
			\multicolumn{1}{r|}{$w_+$} &
                        \multicolumn{1}{r|}{2} \\
      %\cline{6-30}
      \hhline{~~~~~-------------------------}
			$S(\wh{K}_5)(A)$ &
			&
			&
			&
			&
			0 &
			$w_-$ &
			? &
			? &
			? &
			? &
			$\wh w_-$ &
			? &
			? &
			? &
			1 &
			1 &
			? &
			? &
			? &
			$\wh w_+$ &
			? &
			? &
			? &
			? &
			$w_+$ &
			2 &
			&
			&
      \\
      %\cline{3-27}
      \hhline{~~-------------------------}
			$S(\wh{K}_5)$ &
			\multicolumn{1}{l|}{} &
			\multicolumn{1}{r|}{0} &
			\multicolumn{1}{r|}{$w_-$} &
			$w_-$ &
			$w_-$ &
			$w_-$ &
			$z_-$ &
			$\wh w_-$ &
			$\wh w_-$ &
			$\wh w_-$ &
			$\wh w_-$ &
			1 &
			1 &
			1 &
			1 &
			1 &
			$\wh w_+$ &
			$\wh w_+$ &
			$\wh w_+$ &
			$\wh w_+$ &
			$z_+$ &
			$w_+$ &
			$w_+$ &
			$w_+$ &
			$w_+$ &
			2 &
			&
			&
      \\
      %cline{3-27}
      \hhline{~~-------------------------}
    \end{tabular}%
  }
\end{table}
%\vspace{1mm}
\noindent where we denote for simplicity
$ w_\pm=1\pm \left(\frac{1}{2} \sqrt{\frac{1}{5}
    \left(9+\sqrt{21}\right)}\right)$, $\wh w_\pm=1\pm \left(\frac{1}{2}
  \sqrt{\frac{1}{5} \left(9-\sqrt{21}\right)}\right)$ and
$z_\pm=1\pm \sqrt{\frac{2}{5}}$.
This bracketing determines six eigenvalues (highlighted in light gray
in the diagram) where we exploit the fact that $w_\pm$ and
$\wh w_\pm$, both have multiplicity $4-(6-3)=1$ and $1$ has
multiplicity $5-(6-3)=2$.

To determine the remaining $14$ eigenvalues we take $G_2=S(K_6)$ as
an auxiliary graph. In this case the shrinking number is $t_2=1$ and
we obtain the relation
\begin{equation*}
	S(K_6)\less[1] S(\wh{K}_5)(A) \less S(K_6)
\end{equation*}
which can be represented by the following spectral diagram that fixes
the remaining eigenvalues which are highlighted in dark gray.
\vspace{1mm}
\begin{table}[h]
	{%
		\begin{tabular}{lll|
				>{\columncolor[gray]{0.75}}l |
				>{\columncolor[gray]{0.6}}l |
				>{\columncolor[gray]{0.6}}l |
				>{\columncolor[gray]{0.6}}l |
				>{\columncolor[gray]{0.6}}l |
				>{\columncolor[gray]{0.75}}l |
				>{\columncolor[gray]{0.6}}l |
				>{\columncolor[gray]{0.6}}l |
				>{\columncolor[gray]{0.6}}l |
				>{\columncolor[gray]{0.6}}l |
				>{\columncolor[gray]{0.6}}l |
				>{\columncolor[gray]{0.6}}l |
				>{\columncolor[gray]{0.6}}l |
				>{\columncolor[gray]{0.6}}l |
				>{\columncolor[gray]{0.75}}l |
				>{\columncolor[gray]{0.6}}l |
				>{\columncolor[gray]{0.6}}l |
				>{\columncolor[gray]{0.6}}l |
				>{\columncolor[gray]{0.6}}l |
				>{\columncolor[gray]{0.75}}l |l}
%			\cline{4-24}
			\hhline{~~~---------------------}
			$S(K_6)$ &
			&
			&
			0 &
			$z_-$ &
			$z_-$ &
			$z_-$ &
			$z_-$ &
			$z_-$ &
			1 &
			1 &
			1 &
			1 &
			1 &
			1 &
			1 &
			1 &
			1 &
			$z_+$ &
			$z_+$ &
			$z_+$ &
			$z_+$ &
			$z_+$ &
                                \multicolumn{1}{l|}{2} \\
%                  \cline{3-24}
			\hhline{~~----------------------}
			$S(\wh{K}_5)(A)$ &
			\multicolumn{1}{l|}{} &
			0 &
			$w_-$ &
			$z_-$ &
			$z_-$ &
			$z_-$ &
			$z_-$ &
			$\wh w_-$ &
			1 &
			1 &
			1 &
			1 &
			1 &
			1 &
			1 &
			1 &
			$\wh w_+$ &
			$z_+$ &
			$z_+$ &
			$z_+$ &
			$z_+$ &
			$w_+$ &
                                \multicolumn{1}{l|}{2} \\
                  %\cline{3-24}
			\hhline{~~----------------------}
			$S(K_6)$ &
			\multicolumn{1}{l|}{} &
			0 &
			$z_-$ &
			$z_-$ &
			$z_-$ &
			$z_-$ &
			$z_-$ &
			1 &
			1 &
			1 &
			1 &
			1 &
			1 &
			1 &
			1 &
			1 &
			$z_+$ &
			$z_+$ &
			$z_+$ &
			$z_+$ &
			$z_+$ &
			2 &
                  \\
                  %\cline{3-23}
                  \hhline{~~---------------------}
		\end{tabular}%
	}
\end{table}
\vspace{1mm}
Again, only the number pendant paths and the length of the
partitions are relevant.  Therefore, the partitions
$A=\lMset 2, 3 \rMset$ and $B=\lMset 1, 4 \rMset$ give isospectral
graphs $S(\wh{K}_5)(A)$ and $S(\wh{K}_5)(B)$ which are not
isomorphic because the degree lists are different:
\begin{equation*}
  \deg S(\wh{K}_5)(A)=(\underbrace{2,\dots,2}_{16}, 3, 5,5,5,5,5)
  \qquadtext{and}
  \deg S(\wh{K}_5)(B)=(1,\underbrace{2,\dots,2}_{15}, 4, 5,5,5,5,5)
\end{equation*}

%-----------------------------------------------------------------------
\begin{example}
  In a similar vein, the discrete fuzzy bipartite graphs below are
  also isospectral and, clearly, non-isomorphic using the edge
  subdivision for all the edges on the graph (see~
  \Fig{subdivision2.ex}).
  \begin{figure}[h]
    \centering {
      %%%%%%%%%%%%%%%%%%%%%%%%%%%%%%%%%%%%% Part 3
      \begin{tikzpicture}[auto, vertex/.style={circle,draw=black!100,%
          fill=black!100, %thick,
          inner sep=0pt,minimum size=1mm},scale=1] %
        \foreach \r in {1,2,3,4,5} %
        { %
          \node (C\r) at (0,\r/2) [vertex,%fill=red,
          label=below:] {}; %
          \node (b\r) at (1,\r/2) [vertex,label=below:] {};%
          \node (f\r) at (-1,\r/2+.125) [vertex,label=below:] {};%
          \node (g\r) at (-1,\r/2-.125) [vertex,label=below:] {};%

          \path [-](C\r) edge node[right] {} (b\r);%
          \path [-](f\r) edge node[right] {} (C\r);%
          \path [-](g\r) edge node[right] {} (C\r);%
        }%
        \foreach \s in {2,7} %
        { %
          \node (E\s) at (2,\s/4+.25) [vertex,%fill=red,
          label=below:] {}; %
        }
        \foreach \r in {3,4,5}  {\path [-](b\r) edge node[right] {} (E7);}
        \foreach \r in {1,2}  {\path [-](b\r) edge node[right] {} (E2);}

        \foreach \p in {2,4} %
        { %
          \node (D\p) at (-2,\p/2) [vertex,%fill=red,
          label=below:] {}; %
        }%

        \foreach \r in {1,2,3,4,5} %
        { %
          \foreach \p in {2} %
          { %
            \path [-](g\r) edge node[right] {} (D\p);%
          }%
          \foreach \p in {4} %
          { %
            \path [-](f\r) edge node[right] {} (D\p);%
          }%
        }%
        \draw[] (0.35,0)node[left] {\small$S(\wh{K}_{2,5}(A))$};
      \end{tikzpicture}
      \quad \quad \quad
      \begin{tikzpicture}[auto, vertex/.style={circle,draw=black!100,%
          fill=black!100, %thick,
          inner sep=0pt,minimum size=1mm},scale=1] %
        \foreach \r in {1,2,3,4,5} %
        { %
          \node (C\r) at (0,\r/2) [vertex,%fill=red,
          label=below:] {}; %
          \node (b\r) at (1,\r/2) [vertex,label=below:] {};%
          \node (f\r) at (-1,\r/2+.125) [vertex,label=below:] {};%
          \node (g\r) at (-1,\r/2-.125) [vertex,label=below:] {};%

          \path [-](C\r) edge node[right] {} (b\r);%
          \path [-](f\r) edge node[right] {} (C\r);%
          \path [-](g\r) edge node[right] {} (C\r);%
        }%
        \foreach \s in {2,7} %
        { %
          \node (E\s) at (2,\s/4) [vertex,%fill=red,
          label=below:] {}; %
        }
        \foreach \r in {2,3,4,5}  {\path [-](b\r) edge node[right] {} (E7);}
        \foreach \r in {1}  {\path [-](b\r) edge node[right] {} (E2);}

        \foreach \p in {2,4} %
        { %
          \node (D\p) at (-2,\p/2) [vertex,%fill=red,
          label=below:] {}; %
        }%

        \foreach \r in {1,2,3,4,5} %
        { %
          \foreach \p in {2} %
          { %
            \path [-](g\r) edge node[right] {} (D\p);%
          }%
          \foreach \p in {4} %
          { %
            \path [-](f\r) edge node[right] {} (D\p);%
          }%
        }%
        \draw[] (0.35,0)node[left] {\small$S(\wh{K}_{2,5}(B))$};
      \end{tikzpicture}

    } \caption{ Examples of isospectral graphs
      $S(\wh{K}_{2,5}(A))$ and $S(\wh{K}_{2,5}(B))$, obtained
      by edge subdivision of the fuzzy complete bipartite construction
      $\wh{K}_{2,5}(A)$ and $\wh{K}_{2,5}(B)$ for $A=(4,1)$
      and $B=(2,3)$.  }
    \label{fig:subdivision2.ex}
  \end{figure}
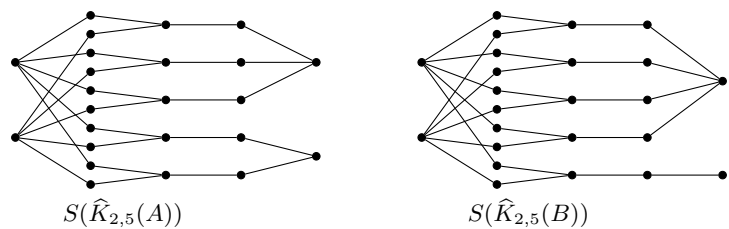

\end{example}

Finally, as in \Cor{metric-1}, we can extend isospectrality to the
metric graph scenario since the subdivision graphs are connected, and the
number of edges and vertices are preserved in the merging along the
different partitions. Hence all the graphs constructed in this section
are also isospectral as equilateral metric graphs for the Kirchhoff Laplacian.
%--------------------------------------------------------

%--------------------------------------------------------
%
% yyyy
% Bibliography
%
%-----------------------------------------------------------------------

%\bibliographystyle{/home/post/Aktuell/BibTeX/my-amsalpha}
%                                % on Olaf's computer
%\bibliography{/home/post/Aktuell/BibTeX/literatur}

%%% include make.bbl here:
\providecommand{\bysame}{\leavevmode\hbox to3em{\hrulefill}\thinspace}
\providecommand{\MR}{\relax\ifhmode\unskip\space\fi MR }
\providecommand{\MRhref}[2]{%
  \href{http://www.ams.org/mathscinet-getitem?mr=#1}{#2} }
\providecommand{\href}[2]{#2}

\end{document}